\documentclass{amsart}
\usepackage{amssymb,amsthm,url, graphicx,color}
\usepackage[labelfont=normal]{subcaption}
\usepackage{hyperref, tikz}
\usepackage{color}
\usetikzlibrary{calc}

\newtheorem{theorem}{Theorem}[section]
\newtheorem{lemma}[theorem]{Lemma}
\newtheorem{corollary}[theorem]{Corollary}

\newtheorem{question}[theorem]{Question}

\theoremstyle{definition}
\newtheorem{construction}[theorem]{Construction}
\newtheorem{definition}[theorem]{Definition}
\newtheorem{example}[theorem]{Example}

\renewcommand{\emph}[1]{{\it{#1}}}

\renewcommand{\phi}{\varphi}
\newcommand{\Dis}{\hbox{\rm Dis}}
\newcommand{\BGCG}{\mathop{\rm{BGCG}}}
\newcommand{\SDD}{\mathop{\rm{SDD}}}
\newcommand{\W}{{\rm{W}}}
\newcommand{\PX}{{\rm{PX}}}
\newcommand{\K}{{\rm{K}}}
\newcommand{\V}{{\rm{V}}}
\newcommand{\E}{{\rm{E}}}
\newcommand{\D}{{\rm{D}}}
\newcommand{\R}{{\rm{R}}}
\newcommand{\CC}{{\rm{C}}}
\newcommand{\Aut}{{\rm Aut}}
\newcommand{\Sym}{{\rm Sym}}

\newcommand{\id}{{\rm id}}

\newcommand{\ZZ}{\mathbb{Z}}

\newcommand{\C}{{\mathcal{C}}}

\newcommand{\Pp}{{\mathcal{P}}}

\newcommand{\EOC}{\hfill $\diamond$}

\newcommand{\Mate}{M}
\newcommand{\Mat}{{\rm Mate}}
\newcommand{\Split}{{\rm Split}}
\newcommand{\Pairs}{{\rm Pairs}}

\title[Base Graph -- Connection Graph]{Base Graph -- Connection Graph: Dissection and  Construction}

\author[P. Poto\v{c}nik]{Primo\v{z} Poto\v{c}nik}
\address{Primo\v{z} Poto\v{c}nik\newline
 Faculty of Mathematics and Physics,
 University of Ljubljana, \newline 
Jadranska 19, 1000 Ljubljana, Slovenia;\\
also affiliated with:
Institute of mathematics, Physics and Mechanics, Jadranska 19, Ljubljana, Slovenia.
}\email{primoz.potocnik@fmf.uni-lj.si}

\author[G. Verret]{Gabriel Verret}
\address{Gabriel Verret\newline
Department of Mathematics, The University of Auckland, \newline Private Bag 92019,  Auckland 1142, New Zealand.}
\email{g.verret@auckland.ac.nz}

\author[S. Wilson]{Stephen Wilson}
\address{Stephen Wilson\newline
Department of Mathematics and Statistics, Northern Arizona University, \newline
Box 5717, Flagstaff, AZ 86011, USA.}
 \email{stephen.wilson@nau.edu}

\thanks{The first author is supported by Slovenian Research Agency, projects J1--1691 and P1--0294. The second author is supported by UWA as part of the Australian Research Council grant DE130101001.}

\begin{document}

\begin{abstract}
This paper presents a phenomenon which sometimes occurs in  tetravalent bipartite locally dart-transitive graphs, called a {\it Base Graph -- Connection Graph} dissection.  In this dissection, each white vertex is split into two vertices of valence 2 so that the connected components of the result are isomorphic.   Given the Base Graph whose subdivision is isomorphic to each component, and the Connection Graph, which describes how the components overlap, we can,   in some cases, provide a construction which can make  a graph having such a decomposition.  This paper investigates the general phenomenon as well as the special cases in which the connection graph has no more than one edge.
\end{abstract}

\maketitle

\section{Introduction}\label{sec:Intro}

We start with an informal introduction to the topic of this paper. Precise definitions and statements will follow in later sections. Our investigation stems from our desire to understand edge-transitive tetravalent graphs and to construct an extensive census of such graphs (see~\cite{C4} and the accompanying paper \cite{recipe}). 

It is easy to see that if $\Gamma$ is a connected graph and $G\le \Aut(\Gamma)$ acts transitively on the edges of $\Gamma$, then either $G$ acts transitively on the vertices of $\Gamma$, or $\Gamma$ is bipartite with each set of the bipartition being an orbit of $G$. It the latter case, we say that $\Gamma$ is {\it $G$-bi-transitive}, or just {\it bi-transitive} if we do not need to mention $G$ explicitly. Observe that in a $G$-bi-transitive graph $\Gamma$, the stabilizer $G_v$ of each vertex $v$ is transitive on the neighbourhood $\Gamma(v)$. For that reason, $G$-bi-transitive graphs are sometimes known as {\it locally $G$-arc-transitive graphs}.

Bi-transitive graphs are one of the most studied highly symmetrical families of graphs, with most of the work concentrating on the case where the action of $G_v$ on $\Gamma(v)$ is quasi-primitive, primitive or even doubly-transitive (see for example \cite{BonStel,feng,GLP,jaj,pablo,erik}).

In this paper we present a construction, called the $\BGCG$ construction, whose  output is a tetravalent graph $\Gamma$ which is $G$-bi-transitive for some $G\le \Aut(\Gamma)$, such that the stabilizer $G_v$ of some vertex $v$ acts imprimitively on $\Gamma(v)$ (we shall call such graphs {\it tetravalent bi-transitive locally imprimitive}).
It transpires that such graphs make up a very important and difficult-to-handle family of tetravalent edge-transitive graphs, hence our interest in them.

The $\BGCG$ construction is comprehensive in the sense that every tetravalent bi-transitive locally imprimitive graph can be constructed by applying the construction to a suitable smaller tetravalent dart-transitive graph. Since bi-transitive graphs are bipartite, we shall also think of them as being properly colored by two colors, black and white. We shall then call them {\it $2$-colored} graphs.

The comprehensiveness of the $\BGCG$ construction is shown by exhibiting its inverse operation, called dissection, which takes a suitable $2$-colored bi-transitive graph as an input and returns a smaller dart-transitive graph. We first start with a few illustrative examples showing the idea behind dissection, while the precise details and results about the two constructions are in Section~\ref{DiS+BGCG}.

Consider a connected tetravalent $2$-colored graph. Suppose that at each white vertex, we separate that vertex and its four incident edges into two vertices, each incident with two of the four edges.  This process is called a {\it dissection} and it may or may not disconnect the graph.  The result is a graph, still bipartite, in which the black vertices have valency $4$ and the white vertices have valency $2$.  The resulting graph is thus a subdivision $X^*$ of some tetravalent (not necessarily connected) graph $X$. If the original graph is $G$-bi-transitive with $G_v$ acting imprimitively on $\Gamma(v)$ for every white vertex $v$ and if the decomposition of the edges incident to $v$ forms a suitably chosen system of imprimitivity of $G_v$, then $X$ will be dart-transitive.
 
\begin{example} \label{ex:first}
 Consider the circulant graph $\CC_{10}(1, 3)$, shown in Figure~\ref{fig:C1013a} as a $2$-colored graph.  If its white vertices are dissected as in Figure~\ref{fig:C1013b}, the result is $\K_5^*$, as seen in Figure~\ref{fig:C1013c}. Each white vertex of $\CC_{10}(1, 3)$ corresponds to a pair of edges of $\K_5$.

Given the correct pairing of edges in $\K_5$, the graph $\CC_{10}(1, 3)$ can be reconstructed. This is done by first subdividing each edge of $\K_5$ to form $\K_5^*$, and then identifying those pairs of new vertices that correspond to paired edges of $\K_5$.  This (re-)construction of $\CC_{10}(1, 3)$ is called a $\BGCG$ construction. (See Example~\ref{example:later}.)
\end{example}

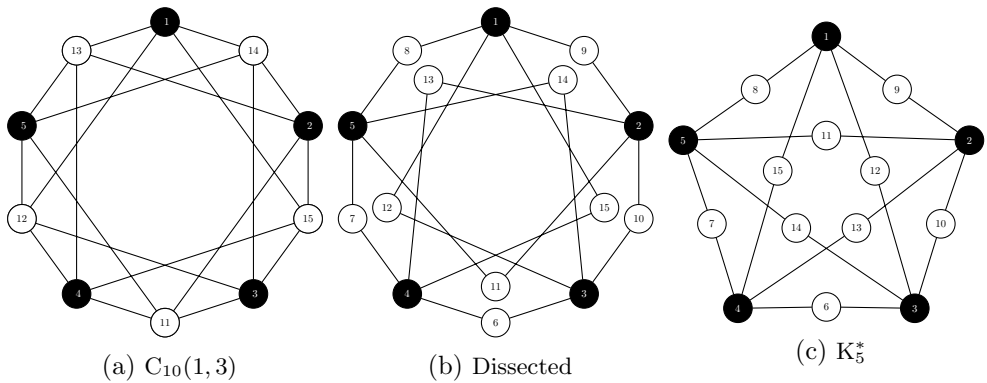
\begin{figure}[hhh]
\centering
\begin{subfigure}{1.7in}

\begin{tikzpicture} [scale=.4, transform shape]
\def\r{5}
\def\s{5}
\def\c{3}
\def\t{5}
\tikzstyle{every node}=[draw,shape=circle, minimum size=27pt, radius= 10, fill = black, text = white];
\node   (1) at (  {\c+\r*sin(0*36)},    {\c+\r*cos(0*36)}  ) {$~1$};
\node (2) at  (  {\c+\r*sin(2*36)},    {\c+\r*cos(2*36)}  ) {$~2$};
\node (3) at  (  {\c+\r*sin(4*36)},    {\c+\r*cos(4*36)}  ) {$~3$};
\node (4) at  (  {\c+\r*sin(6*36)},    {\c+\r*cos(6*36)}  ) {$~4$};
\node (5) at  (  {\c+\r*sin(8*36)},    {\c+\r*cos(8*36)}  ) {$~5$};
\tikzstyle{every node}=[draw,shape=circle, minimum size=27pt,  radius= 10, fill = white, text = black];
\node (6) at  (  {\c+\t*sin(5*36)},    {\c+\t*cos(5*36)}  ) {$~6$};
\node (7) at  (  {\c+\t*sin(7*36)},    {\c+\t*cos(7*36)}  ) {$~7$};
\node (8) at  (  {\c+\t*sin(9*36)},    {\c+\t*cos(9*36)}  ) {$~8$};
\node (9) at  (  {\c+\t*sin(1*36)},    {\c+\t*cos(1*36)}  ) {$~9$};
\node (10) at  (  {\c+\t*sin(3*36)},    {\c+\t*cos(3*36)}  ) {$10$};
\node (11) at (  {\c+\s*sin(5*36)},    {\c+\s*cos(5*36)}  )  {$11$};
\node (12) at  (  {\c+\s*sin(7*36)},    {\c+\s*cos(7*36)}  ) {$12$};
\node (13) at   (  {\c+\s*sin(9*36)},    {\c+\s*cos(9*36)}  ) {$13$};
\node (14) at  (  {\c+\s*sin(1*36)},    {\c+\s*cos(1*36)}  ) {$14$};
\node (15) at  (  {\c+\s*sin(3*36)},    {\c+\s*cos(3*36)}  ) {$15$};
\draw   (1)--(9)--(2)--(10)--(3)--(6)--(4)--(7)--(5)--(8)--(1);
\draw  (11)--(5)--(14)--(3)--(12)--(1)--(15)--(4)--(13)--(2)--(11);
\end{tikzpicture} 

\caption{$\CC_{10}(1, 3)$}\label{fig:C1013a}
\end{subfigure}
~
\begin{subfigure}{1.7in}
\begin{tikzpicture} [scale=.4, transform shape]
\def\r{5}
\def\t{5}
\def\s{3.8}
\def\c{3}
\tikzstyle{every node}=[draw,shape=circle, minimum size=27pt, radius= 1.5, fill = black, text = white];
\node   (1) at (  {\c+\r*sin(0*36)},    {\c+\r*cos(0*36)}  ) {$ 1$};
\node (2) at  (  {\c+\r*sin(2*36)},    {\c+\r*cos(2*36)}  ) {$2$};
\node (3) at  (  {\c+\r*sin(4*36)},    {\c+\r*cos(4*36)}  ) {$3$};
\node (4) at  (  {\c+\r*sin(6*36)},    {\c+\r*cos(6*36)}  ) {$4$};
\node (5) at  (  {\c+\r*sin(8*36)},    {\c+\r*cos(8*36)}  ) {$5$};
\tikzstyle{every node}=[draw,shape=circle,minimum size=27pt,  radius= 1.5, fill = white, text = black];
\node (6) at  (  {\c+\t*sin(5*36)},    {\c+\t*cos(5*36)}  ) {$6$};
\node (7) at  (  {\c+\t*sin(7*36)},    {\c+\t*cos(7*36)}  ) {$7$};
\node (8) at  (  {\c+\t*sin(9*36)},    {\c+\t*cos(9*36)}  ) {$8$};
\node (9) at  (  {\c+\t*sin(1*36)},    {\c+\t*cos(1*36)}  ) {$9$};
\node (10) at  (  {\c+\t*sin(3*36)},    {\c+\t*cos(3*36)}  ) {$10$};
\node (11) at (  {\c+\s*sin(5*36)},    {\c+\s*cos(5*36)}  )  {$11$};
\node (12) at  (  {\c+\s*sin(7*36)},    {\c+\s*cos(7*36)}  ) {$12$};
\node (13) at   (  {\c+\s*sin(9*36)},    {\c+\s*cos(9*36)}  ) {$13$};
\node (14) at  (  {\c+\s*sin(1*36)},    {\c+\s*cos(1*36)}  ) {$14$};
\node (15) at  (  {\c+\s*sin(3*36)},    {\c+\s*cos(3*36)}  ) {$15$};
\draw   (1)--(9)--(2)--(10)--(3)--(6)--(4)--(7)--(5)--(8)--(1);
\draw  (11)--(5)--(14)--(3)--(12)--(1)--(15)--(4)--(13)--(2)--(11);
\end{tikzpicture} 
\caption{Dissected}\label{fig:C1013b}
\end{subfigure}
~
\begin{subfigure}{1.7in}

\begin{tikzpicture} [scale=.4, transform shape]
\def\r{5}
\def\t{4}
\def\s{-1.7}
\def\c{3}
\tikzstyle{every node}=[draw,shape=circle, minimum size=27pt, radius= 1.5, fill = black, text = white];
\node   (1) at (  {\c+\r*sin(0*36)},    {\c+\r*cos(0*36)}  ) {$ 1$};
\node (2) at  (  {\c+\r*sin(2*36)},    {\c+\r*cos(2*36)}  ) {$2$};
\node (3) at  (  {\c+\r*sin(4*36)},    {\c+\r*cos(4*36)}  ) {$3$};
\node (4) at  (  {\c+\r*sin(6*36)},    {\c+\r*cos(6*36)}  ) {$4$};
\node (5) at  (  {\c+\r*sin(8*36)},    {\c+\r*cos(8*36)}  ) {$5$};
\tikzstyle{every node}=[draw,shape=circle, minimum size=27pt, radius= 1.5, fill = white, text = black];
\node (6) at  (  {\c+\t*sin(5*36)},    {\c+\t*cos(5*36)}  ) {$6$};
\node (7) at  (  {\c+\t*sin(7*36)},    {\c+\t*cos(7*36)}  ) {$7$};
\node (8) at  (  {\c+\t*sin(9*36)},    {\c+\t*cos(9*36)}  ) {$8$};
\node (9) at  (  {\c+\t*sin(1*36)},    {\c+\t*cos(1*36)}  ) {$9$};
\node (10) at  (  {\c+\t*sin(3*36)},    {\c+\t*cos(3*36)}  ) {$10$};
\node (11) at (  {\c+\s*sin(5*36)},    {\c+\s*cos(5*36)}  )  {$11$};
\node (12) at  (  {\c+\s*sin(7*36)},    {\c+\s*cos(7*36)}  ) {$12$};
\node (13) at   (  {\c+\s*sin(9*36)},    {\c+\s*cos(9*36)}  ) {$13$};
\node (14) at  (  {\c+\s*sin(1*36)},    {\c+\s*cos(1*36)}  ) {$14$};
\node (15) at  (  {\c+\s*sin(3*36)},    {\c+\s*cos(3*36)}  ) {$15$};
\draw   (1)--(9)--(2)--(10)--(3)--(6)--(4)--(7)--(5)--(8)--(1);
\draw  (11)--(5)--(14)--(3)--(12)--(1)--(15)--(4)--(13)--(2)--(11);
\end{tikzpicture} 
\caption{$\K_5^*$}\label{fig:C1013c}
\end{subfigure}
\caption{Dissection of $\CC_{10}(1, 3)$}\label{fig:C1013}
\end{figure}
  
\begin{example} \label{ex:second}
Consider the rose window graph $\R_{10}(4,1)$ (see \cite{Rose} for the definition of rose window graphs), shown in Figure \ref{fig:R1041a}.  If its white vertices are dissected as in Figure \ref{fig:R1041b}, the resulting graph has two components, as can be seen in Figure \ref{fig:R1041c}, and each of these is isomorphic to $\K_5^*$, as in Figure \ref{fig:R1041d}.
\end{example}

\begin{figure}[hhh]
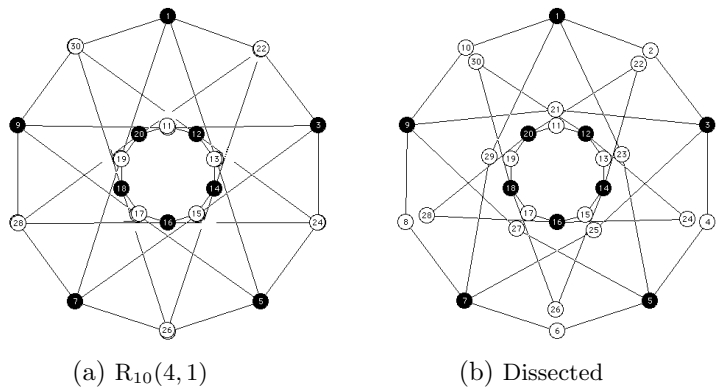

\begin{subfigure}{2in}
\includegraphics[height=45mm]{R1041a.pdf}
\caption{$\R_{10}(4, 1)$}\label{fig:R1041a}
\end{subfigure}
~
\begin{subfigure}{2in}
\includegraphics[height=45mm]{R1041b.pdf}
\caption{Dissected}\label{fig:R1041b}
\end{subfigure}
\label{fig:R1041}
\caption{Dissection of $\R_{10}(4, 1)$}\label{fig:C1041}
\end{figure}

\begin{figure}[hhh]
\centering
\includegraphics[height=45mm]{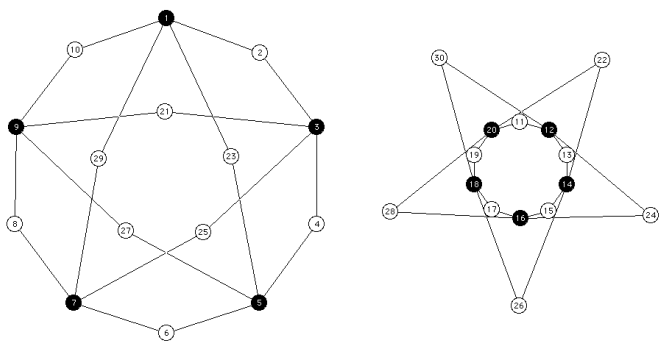}
\caption{The components separated}\label{fig:R1041c}
\end{figure}

\begin{figure}[hhh]
\begin{center}
\includegraphics[height=49mm]{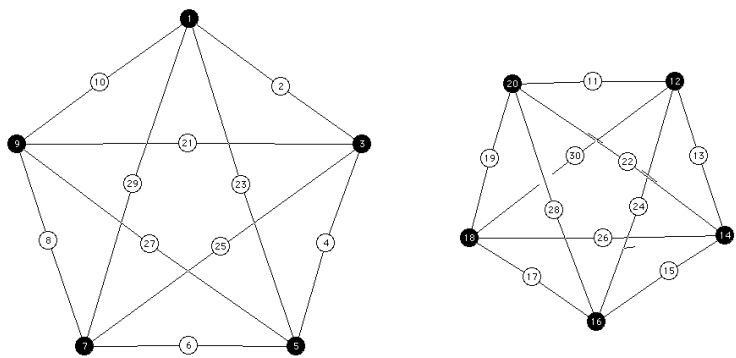}
\caption{The components as copies of $\K_5^*$}\label{fig:R1041d}
\end{center}

\end{figure}

The $\BGCG$ construction and the corresponding dissection will be defined and discussed in precise terms in Section~\ref{DiS+BGCG}.  After a section of preliminary matters, we will  present these two notions in a broader setting of bi-transitive graphs of arbitrary valence.

\section{Preliminaries}

Although we are mainly interested in simple graphs, it will be convenient to allow graphs to have parallel edges. Let us start this section with a quick overview of notions pertaining to such graphs.

In this paper, a \emph{graph} will be a triple $\Gamma=(V,E,\partial)$ where $V$ and $E$ are finite sets and $\partial$ is a mapping from $E$ to ${\V(\Gamma) \choose 2} = \{ X\subseteq \V(\Gamma) \mid |X|=2\}$.  The set $V$ will be called the \emph{vertex-set} of $\Gamma$ and denoted $\V(\Gamma)$, while its elements are called \emph{vertices}. The set $E$ will be called the \emph{edge-set} of $\Gamma$ and denoted $\E(\Gamma)$, while its elements are called \emph{edges}. For an edge $e$, the elements of $\partial(e)$ are called {\it endvertices} of $e$ and two vertices $u,v \in \V(\Gamma)$ are called {\it adjacent} provided there exists $e\in\E(\Gamma)$ such that $\partial(e) = \{u,v\}$. Similarly, two edges $e_1,e_2\in\E(\Gamma)$ are {\it adjacent} if $\partial(e_1)\cap \partial(e_2) \not = \emptyset$. Two edges $e_1, e_2 \in \E(\Gamma)$ with $\partial(e_1) = \partial(e_2)$ are called {\it parallel}.  A graph $\Gamma$ is {\it simple} provided it has no parallel edges or, equivalently, provided the function $\partial$ is injective. In this case, $\E(\Gamma)$ can be identified with a subset of ${\V(\Gamma) \choose 2}$. A simple graph will thus often be given in the usual way, as a pair $(V,E)$ where $E\subseteq {\V(\Gamma) \choose 2}$.

A {\it dart} of $\Gamma$ (sometimes called an arc) is an ordered pair $(e,v)$ where $e \in \E(\Gamma)$ 
is its {\it underlying edge} and $v\in \partial(e)$ is its {\it initial vertex}.
 In a simple graph, a dart $(e,v)$ is uniquely determined by the ordered pair 
 $(v,u)$,
 where $\partial(e) =\{u,v\}$. We will thus often refer to darts in simple graphs as ordered pairs of adjacent vertices.

The {\it neighbourhood} of a vertex $v$ in $\Gamma$, denoted $\Gamma(v)$, is the set of edges $e$ of $\Gamma$ such that $v\in\partial(e)$, and the \emph{valency} of  $v$ is the cardinality of $\Gamma(v)$. If $\Gamma$ is simple, then every edge in $\Gamma(v)$ determines uniquely an adjacent vertex of $v$, hence in this case, $\Gamma(v)$ can be interpreted as the set of vertices adjacent to $v$. A graph will be called \emph{$k$-valent} if all of its vertices have valency $k$. A $4$-valent graph will sometimes be called \emph{tetravalent}. 

A {\it symmetry} $g$ of a graph $\Gamma$ (sometimes called an automorphism) is a permutation of $\V(\Gamma)\cup\E(\Gamma)$ which preserves $\V(\Gamma)$ and $\E(\Gamma)$ and such that, for every $e\in E$, $\partial(e^g)=\partial(e)^g$. Note that we use superscript notation for the action of a permutation:  if  $\Omega$ is a set, if $\omega \in \Omega$, and if $\alpha \in \Sym(\Omega)$, then  $\omega^{\alpha}$ is the image of $\omega$ under $\alpha$.  
It is then convenient to define the product $\alpha\beta$ so that $\omega^{\alpha\beta} = (\omega^{\alpha})^\beta$.
The symmetries of $\Gamma$ form a group under 
this product, denoted $\Aut(\Gamma)$. If $\Gamma$ is simple, then a symmetry of $\Gamma$ is uniquely determined by its action on $\V(\Gamma)$. In this case, we will often think of $\Aut(\Gamma)$ as a group of permutations on $\V(\Gamma)$.

There are obvious induced actions of $\Aut(\Gamma)$ on the vertices, edges and darts of $\Gamma$. These actions are not necessarily faithful. If $G\leq\Aut(\Gamma)$, we say that $\Gamma$ is $G$-vertex- (or $G$-edge- or $G$-dart-) transitive provided that $G$ acts transitively on the vertices (or edges or darts).  When $G = \Aut(\Gamma)$, the prefix $G$ in the above notation is sometimes omitted.

A {\it $2$-colored} graph is a graph together with a proper coloring of its vertices in black and white (that is, each edge has one black and one white endvertex). Clearly, a $2$-colored graph is bipartite. Note that, if a bipartite graph is connected, then, up to permuting the colors, it admits a unique $2$-coloring and hence $\Aut(\Gamma)$ must preserve the partition of $V(\Gamma)$ into colors. A $2$-colored graph $\Gamma$ will be called {\it $G$-bi-transitive} if it is $G$-edge-transitive and $G$ is a color-preserving group of symmetries of $\Gamma$. A $2$-colored graph $\Gamma$ will be called simply {\it bi-transitive} if there exists a $G$ such that $\Gamma$ is $G$-bi-transitive.  If $\Gamma$ is regular     (i.e. all vertices have the same valency), bi-transitive and has no symmetry which reverses color, we say the graph is {\it semisymmetric}.

To {\it subdivide} an edge $e$ with endpoints $\{u, v\}$ is to remove $e$, introduce a 
 vertex $z$ and two new edges: $e_1$ with endpoints $\{u, z\}$ and $e_2$ with endpoints $\{z, v\}$.     If $X$ is a graph, then the {\it subdivision} of $X$, denoted $X^*$, is a $2$-colored graph formed from $X$ by coloring each vertex of $X$ black and then subdividing each edge of $X$ and coloring the new vertices white.
 
The {\it subdivided double} of $X$, denoted $\SDD(X)$, is formed from $X^*$ by ``doubling'' every black vertex, that is, by replacing every black vertex by two new black vertices, each inheriting all the neighbours of the old vertex.  If $X$ is     simple, tetravalent and dart-transitive, then $\SDD(X)$ is always semisymmetric 
(see \cite{PotWilg4} for details).

\section{General dissection and $\BGCG$ construction}
 \label{DiS+BGCG}

While our primary interest is in the BGCG construction applied to tetravalent graphs, we first start with a more general treatment. Let $\Gamma$ be a $2$-colored graph in which every white vertex has even valence. A {\it split at white vertices} of $\Gamma$ is a partition of $\E(\Gamma)$ such that each part consists of exactly two edges and these two edges have a white endvertex in common. If $X$ is a graph, a {\it separating} relation for $X$ is an equivalence relation $\Mate$ on $\E(X)$ such that no two adjacent edges of $X$ are $\Mate$-related.  

\begin{construction}[{\sc Dissection}]\label{cons:dis}
The input of this construction is a pair $(\Gamma,\Delta)$ such that $\Gamma$ is a simple $2$-colored graph with no isolated white vertices and $\Delta$ is a split at white vertices of $\Gamma$. The output is a pair $(X,\Mate)$ where $X=\Dis(\Gamma,\Delta)$ is a graph and $\Mate=\Mat(\Gamma,\Delta)$ is a separating relation for $X$.

The vertex-set of $X$ is the set of black vertices of $\Gamma$. For each part $\{e_1,e_2\}$ of $\Delta$, we create an edge $e$ in $X$ such that $\delta(e)$ consists of the black endvertex of $e_1$ and the black endvertex of $e_2$. Note that, since $\Gamma$ is simple, these two black vertices are distinct. We shall say that the edge $e$ of $X$ \emph{arises} from $v$, where $v$ is the (white) common endvertex of $e_1$ and $e_2$ in $\Gamma$. We say that two edges of $X$ are $\Mate$-related whenever they arise from the same white vertex of $\Gamma$. This is clearly an equivalence relation on $\E(X)$. Since $\Gamma$ is simple, it follows that $\Mate$ is separating. Moreover, if $v$ is a vertex of $X$, then $v$ has the same valency in $X$ as it has in $\Gamma$.
\EOC
\end{construction}

Note that, if $(X,\Mate) = (\Dis(\Gamma,\Delta), \Mat(\Gamma,\Delta))$, then $X$ might be disconnected even if $\Gamma$ is connected and, moreover,  if $\Gamma$ has 4-cycles, then $X$ may not be simple, as in the next example. 

\begin{example}
This  example concerns
 the \emph{wreath graph} $\W(n, 2)$, which is the simple graph with vertex set  $\ZZ_n \times \ZZ_2$, and $(i,j)$ and $(k,l)$ adjacent if and only if $i-k =\pm 1$.  This family of graphs  will be of some importance to us later.  Let $\Gamma=\W(6,2)$, which is shown in Figure \ref{fig:W62}.

\begin{figure}[hhh]
\centering
\includegraphics[height=35mm]{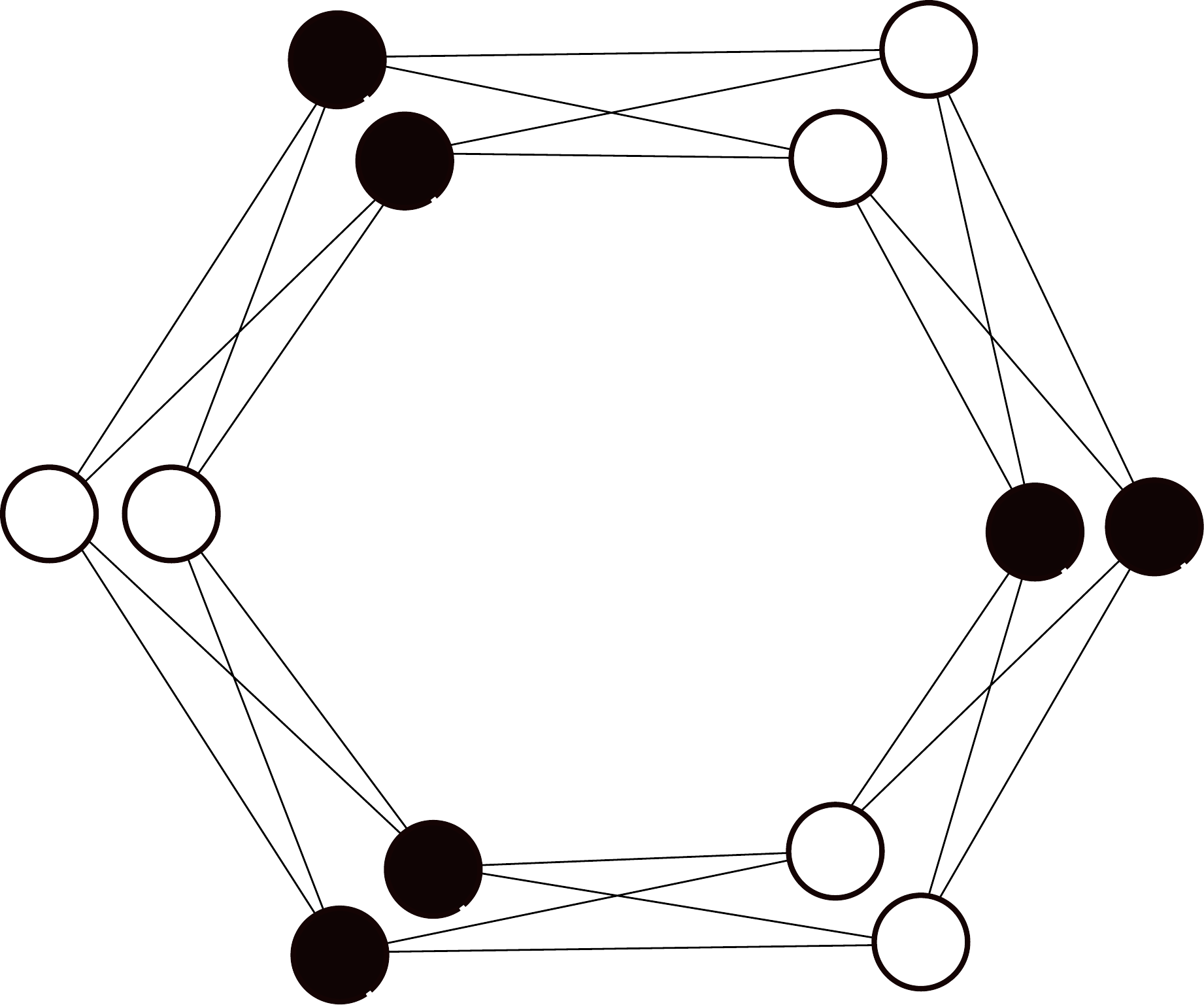}
\caption{$\Gamma=\W(6,2)$}\label{fig:W62}
\end{figure}

We will define two relations on $E(\W(6,2))$.  In the first of these two relations, we say that two edges of $\Gamma$ are related if they have a white endvertex in common and their other endvertices have the same first coordinate. This equivalence relation induces a split $\Delta_1$ of $\Gamma$. The corresponding dissection (in fact, its subdivision) is shown in Figure \ref{fig:W62a}. Note that $\Dis(\Gamma,\Delta_1)$ has three components, each isomorphic to the {\it dipole} with $4$ parallel edges. 

\begin{figure}[hhh]
\centering
\includegraphics[height=35mm]{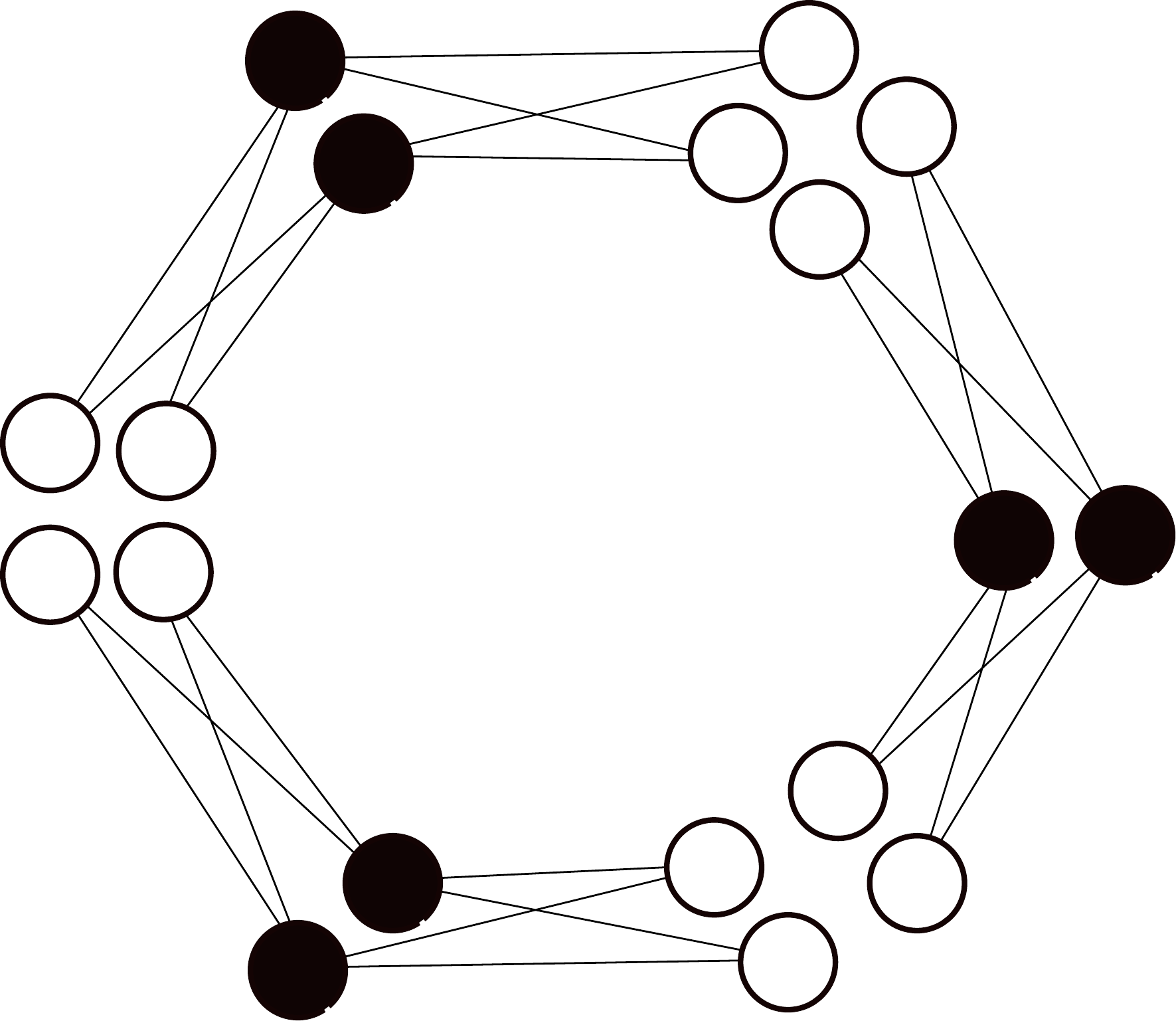}
\caption{Subdivision of $\Dis(\Gamma,\Delta_1)$}\label{fig:W62a}
\end{figure}

Another option is to declare two edges of $\Gamma$ related if they have a white endvertex in common and their other ends have the same second coordinate. This equivalence relation induces a split $\Delta_2$ of $\Gamma$ and the subdivision of the corresponding dissection is shown in Figure~\ref{fig:W62b}. In this case, $\Dis(\Gamma,\Delta_2)$ has two components, each a {\it doubled $3$-cycle}, that is, a simple $3$-cycle with every edge replaced by two parallel edges.

\begin{figure}[hhh]
\centering
\includegraphics[height=35mm]{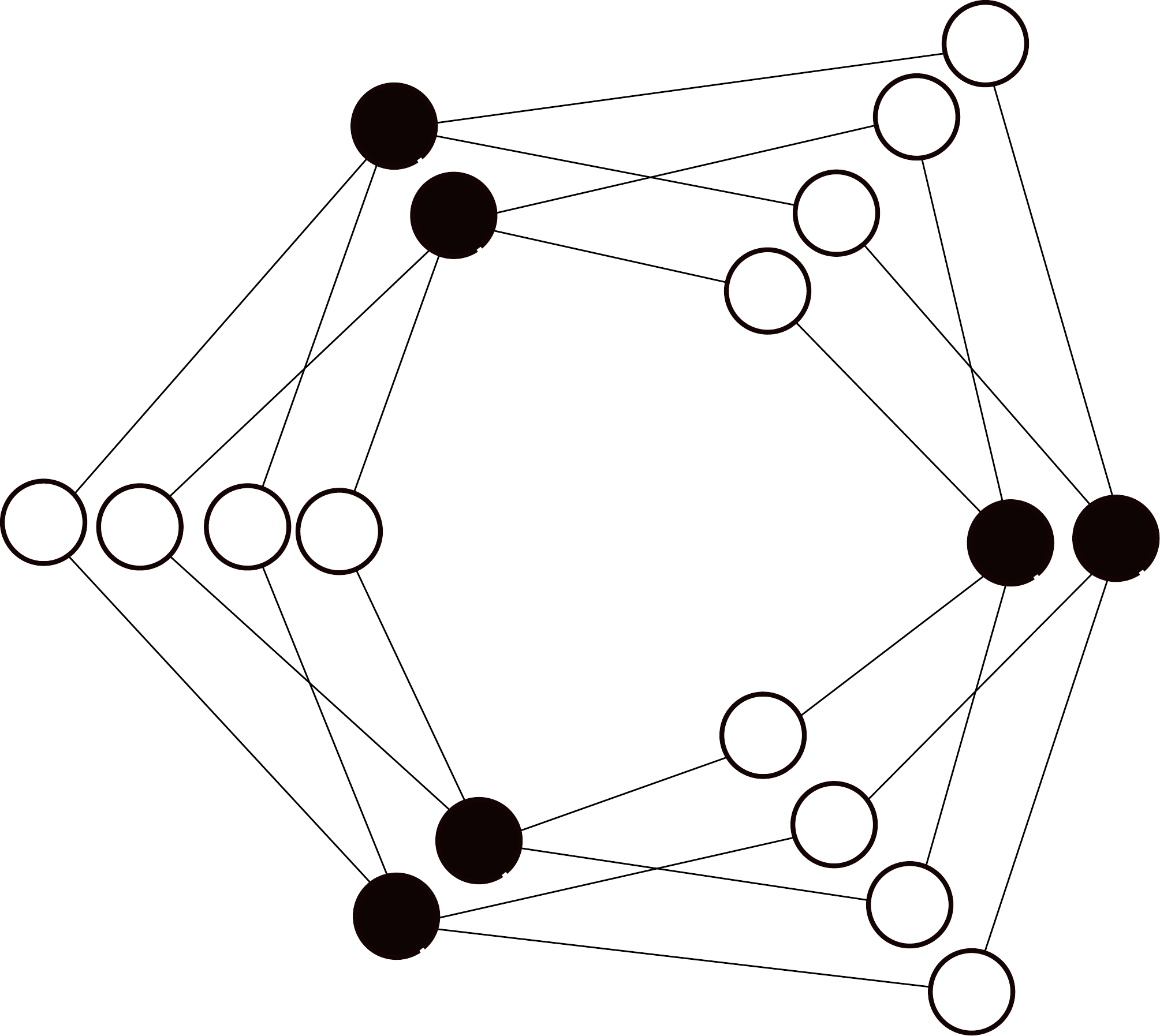}
\caption{Subdivision of $\Dis(\Gamma,\Delta_2)$}\label{fig:W62b}
\end{figure}

\end{example}

Let us now present a construction, which, as will be shown in Lemma~\ref{lem:inverse}, is in some sense 
 an inverse of Construction~\ref{cons:dis}.

\begin{construction}[{\sc BGCG}]\label{cons:BGCG}

The input of this construction is a pair $(X,\Mate)$ where $X$ is a graph and $\Mate$ is a separating relation for $X$. The output of this construction is a pair $(\Gamma,\Delta)$ where $\Gamma=\BGCG(X,\Mate)$ is a simple $2$-colored graph with no isolated white vertices and $\Delta=\Split(X,\Mate)$ is a split at white vertices of $\Gamma$.

Consider the subdivision  $X^*$   of $X$. Note that every black vertex of $X^*$ has the same valency in $X^*$ as it has in $X$ and that every white vertex of $X^*$ has valency $2$.

 Since there is a one--to--one correspondence between the edges of $X$ and the white vertices of $X^*$, the equivalence relation $\Mate$ on $\E(X)$ can be considered as an equivalence relation on the white vertices of $X^*$. We extend $\Mate$ to an equivalence relation on $\V(X^*)$ by declaring that a black vertex of $X^*$ is $\Mate$-related to itself only. For a vertex $v$ of $X^*$, let $[v]$ denote its $\Mate$-equivalence class.

Let $\Gamma=X^*/\Mate$, by which we mean the quotient graph of $X^*$ with respect to the equivalence relation $\Mate$. Explicitly, vertices of $\Gamma$ are equivalence classes of $\Mate$ and, for each edge $e$ of $X^*$ with $\partial(e)=\{u,v\}$, we create an edge $[e]$ of $\Gamma$ with 
$\partial([e])= \{[u],[v]\}$. 

Since the $\Mate$-equivalence classes on $\V(X^*)$ are monochromatic, $\Gamma$ inherits a $2$-coloring from $X^*$ in a natural way. Clearly, $\Gamma$ has no isolated white vertices. Moreover, since $\Mate$ is a separating relation for $X$, it follows that $\Gamma$ is simple and that black vertices of $\Gamma$ have the same valency as they have in $X$. We now define $\Delta$, a split at white vertices of $\Gamma$. Say that two edges of $X^*$ are $R$-related if they have a white endvertex in common. This is clearly an equivalence relation and, since white vertices of $X^*$ have valency $2$, the $R$-equivalence classes have cardinality two. By definition, each edge of $\Gamma$ was induced by an edge of $X^*$ and hence $R$ induces an equivalence relation on the edges of $\Gamma$ which is a split at white vertices of $\Gamma$.
\EOC
\end{construction}

\begin{lemma}\label{lem:inverse}
Let $\Gamma$ be a simple $2$-colored graph with no isolated white vertices and let $\Delta$ be a split at white vertices of $\Gamma$. 
Let $(X,\Mate) = (\Dis(\Gamma,\Delta), \Mat(\Gamma,\Delta))$. Then $\BGCG(X,\Mate) \cong \Gamma$.
\end{lemma}

\begin{proof}
Let us first define a function $\varphi$ from the vertex set of $\BGCG(X,\Mate)$ to the vertex set of $\Gamma$. Observe first that, by definition, the black vertices of $\BGCG(X,\Mate)$ are precisely the black vertices of $X = \Dis(\Gamma,\Delta)$, which are precisely the black vertices of $\Gamma$. Hence we may let the restriction of $\varphi$ on the black vertices of $\BGCG(X,\Mate)$ to be the identity function.

Note that, since $\Gamma$ has no isolated white vertices, there is a natural bijective correspondence between the $\Mate$-equivalence classes and the white vertices of $\Gamma$. On the other hand, a white vertex $v$ of $\BGCG(X,\Mate)$ can be viewed as an  $\Mate$-equivalence class of the edges of $X$, which then corresponds to a white vertex of $\Gamma$. We shall define this white vertex of $\Gamma$ to be the $\varphi$-image of $v$. It is now a matter of straightforward computation to check that $\varphi$ is an isomorphism of graphs.
\end{proof}

Let us now consider the symmetries of the constructed graphs.

\begin{lemma}\label{lemma:yo1}
Let $X$ be  a graph, let $G\leq\Aut(X)$ and let $\Mate$ be a $G$-invariant separating relation for $X$. Then there is a natural action of $G$ as a color-preserving group of symmetries of $\BGCG(X,\Mate)$ and this action is faithful on vertices. Moreover, if $X$ is $G$-dart-transitive, then $\BGCG(X,\Mate)$ is $G$-bi-transitive.
\end{lemma}
\begin{proof}
Let $\Gamma=\BGCG(X,\Mate)$ and let $g\in G$. Note that $\Gamma$ is simple and hence a symmetry of $\Gamma$ is uniquely determined by its action of $\V(\Gamma)$. The black vertices of $\Gamma$ correspond to the vertices of $X$, while the white vertices of $\Gamma$ correspond to $\Mate$-equivalence class of $E(X)$. Since $\Mate$ is $G$-invariant, this shows that there is a natural color-preserving action of $G$ on $\Gamma$. We show that this action is faithful. Let $g\in G$ such that $g$ fixes all vertices of $\Gamma$. It follows that $g$ fixes all black vertices of $X$ and that $g$ also fixes the $\Mate$-equivalence classes setwise. Since no two adjacent edges of $X$ are $\Mate$-related, it follows that $g=1$. Finally, by definition, edges of $\Gamma$ are induced by darts of $X$ from which the last claim follows.
\end{proof}

It is also true that $\Gamma = \BGCG(X, M)$ may have ``unexpected'' symmetries, i.e., symmetries that do not correspond to symmetries of $X$.  In particular, $\Gamma$ might have symmetries which reverse the colors of the vertices.

\begin{lemma}\label{lemma:yo}
Let $\Gamma$ be a simple $2$-colored graph, let $G$ be a group of color-preserving symmetries of $\Gamma$, and let $\Delta$ be a $G$-invariant split at white vertices of $\Gamma$. Let  $(X,\Mate) = (\Dis(\Gamma,\Delta), \Mat(\Gamma,\Delta))$. Then there is an action of $G$ as a group of symmetries of $X$, such that $\Mate$ is $G$-invariant. Moreover, if $\Gamma$ is $G$-bi-transitive, then $X$ is $G$-dart-transitive.
\end{lemma}
\begin{proof}
Since $\Gamma$ is simple, we can think of $G$ as a permutation group on $\V(\Gamma)$. Let $g\in G$ and let $v\in\V(X)$. Recall that $\V(X)$ is the set of black vertices of $\Gamma$ and hence $v^g$ is well-defined. Now, let $e\in\E(X)$ and let $a$ and $b$ be the endvertices of $e$. The edge $e$ of $X$ arose from a white vertex $w$ of $\Gamma$ and $\{\{a,w\},\{b,w\}\}\in\Delta$. It follows that $w^g$ is a white vertex of $\Gamma$ and, since $\Delta$ is $G$-invariant, that $\{\{a^g,w^g\},\{b^g,w^g\}\}\in\Delta$. Define $e^g$ to be the edge of $X$ which arose from $w^g$ corresponding to $\{a^g,b^g\}$. It is not too hard to check that $g$ is a symmetry of $X$ and that the only element of $G$ which fixes every vertex and every edge of $X$ is the identity. In particular, $G$ acts as a group of symmetries of $X$.  Since two edges of $X$ are $\Mate$-related whenever they arose from the same white vertex of $\Gamma$, it follows that $\Mate$ is $G$-invariant. Finally, by definition, each edge of $X$ corresponds to an element of $\Delta$ which is a pair of edges of $\Gamma$. By viewing an edge of $X$ as a pair of darts, we obtain a natural correspondence between darts of $X$ and edges of $\Gamma$. It follows that if $\Gamma$ is $G$-edge-transitive, then $X$ is $G$-dart-transitive.
\end{proof}

\section{Locally imprimitive bi-transitive graphs}

Let $\Gamma$ be  a simple graph, let $v$ be a vertex of $\Gamma$ and let $G\leq\Aut(\Gamma)$. 
We denote by $G_v$ the stabilizer of the vertex $v$ in $G$ and by  $G_v^{\Gamma(v)}$  the permutation group induced by the action of $G_v$ on $\Gamma(v)$. If $\Gamma$ is $G$-bi-transitive, then $G_v^{\Gamma(v)}$ is transitive. If, in addition, $G_v^{\Gamma(v)}$ is imprimitive for some vertex $v$, then we say that $\Gamma$ is {\it $G$-locally imprimitive}.

As mentioned in the introduction, our original interest in the $\BGCG$ construction came from our desire to understand tetravalent $G$-locally imprimitive graphs.
We start with Theorem~\ref {theo:equivalent2} where we prove that a $\BGCG$ construction, when applied to a $G$-dart-transitive graph admitting an appropriate $G$-invariant  separating relation, indeed yields a $G$-locally imprimitive graph. 
We then continue with Theorem~\ref{theo:equivalent} where we prove that every $G$-locally imprimitive graph can be obtained in that way.

Even though later we focus exclusively on tetravalent graphs, 
these two theorems apply to bi-transitive graphs of possibly higher valences.

\begin{theorem}
\label{theo:equivalent2}
Let $X$ be a $k$-valent $G$-dart-transitive graph and let $\Mate$ be a $G$-invariant separating relation for $X$ such that each $\Mate$-equivalence class has cardinality $d$. 
Let $\Gamma=\BGCG(X, \Mate)$. Then $\Gamma$ is a simple $G$-bi-transitive $2$-colored graph with black vertices having valency $k$ and white vertices having valency $2d$. Moreover, for every white vertex $v$ of $\Gamma$, the group $G_v^{\Gamma(v)}$ is imprimitive. 
\end{theorem}

\begin{proof}
Let $\Delta= \Split(X, \Mate)$.
As we saw in Construction~\ref{cons:BGCG}, $\Gamma$ is a simple $2$-colored graph with no isolated white vertices and with black vertices having valency $k$, while $\Delta$ is a split at white vertices of $\Gamma$. By Lemma~\ref{lemma:yo1}, there is a natural bi-transitive action of $G$ on $\Gamma$. Finally, since $\Mate$-equivalence classes have cardinality $d$, it follows that any white vertex $v$ of $\Gamma$ has valency $2d$ and that $G_v^{\Gamma(v)}$ is imprimitive with $d$ blocks of size $2$. 
\end{proof}

\begin{example}
Let $H$ be the Heawood graph, shown in Figure~\ref{Hea}. Let $M$ be the equivalence relation on edges of $H$ of being parallel in this diagram of $H$. Clearly, $M$ is a separating relation and $M$-equivalence classes have cardinality three. (One such class is $\{\{14, 1\}, \{10, 5\}, \{8, 7\}\}$.)  The group of symmetries of $H$ which preserve $M$ is transitive on darts hence, by Theorem~\ref{theo:equivalent2}, $\BGCG(H, \Mate)$ is a simple bi-transitive $2$-colored graph with $14$ black vertices having valency $3$ and $7$ white vertices having valency $6$.   The graph $\BGCG(H, \Mate)$ also arises from the chiral map $\{3,6\}_{2, 1}$ on the torus (the  triangle embedding of $K_7$ in the torus, see \cite{CM}); vertices of the graph are vertices and faces of the map, and two vertices of the graph  are adjacent when the corresponding face and vertex of the map are incident.
\end{example}

\begin{figure}[hhh]
\begin{center}
\begin{tikzpicture} [scale=.4]
\def\r{5}
\def\s{5}
\def\c{3}
\def\t{180/7}
\tikzstyle{every node}=[draw,shape=circle, minimum size=20pt, fill = white, text = black];
\node   (1) at (  {\c+\r*sin(1*\t-.5*\t)},    {\c+\r*cos(1*\t-.5*\t)}  ) {$ 1$};
\node (2) at  (  {\c+\r*sin(2*\t-.5*\t)},    {\c+\r*cos(2*\t-.5*\t)}  ) {$2$};
\node (3) at  (  {\c+\r*sin(3*\t-.5*\t)},    {\c+\r*cos(3*\t-.5*\t)}  ) {$3$};
\node (4) at  (  {\c+\r*sin(4*\t-.5*\t)},    {\c+\r*cos(4*\t-.5*\t)}  ) {$4$};
\node (5) at  (  {\c+\r*sin(5*\t-.5*\t)},    {\c+\r*cos(5*\t-.5*\t)}  ) {$5$};
\node (6) at  (  {\c+\r*sin(6*\t-.5*\t)},    {\c+\r*cos(6*\t-.5*\t)}  ) {$6$};
\node (7) at  (  {\c+\r*sin(7*\t-.5*\t)},    {\c+\r*cos(7*\t-.5*\t)}  ) {$7$};
\node (8) at  (  {\c+\r*sin(8*\t-.5*\t)},    {\c+\r*cos(8*\t-.5*\t)}  ) {$8$};
\node (9) at  (  {\c+\r*sin(9*\t-.5*\t)},    {\c+\r*cos(9*\t-.5*\t)}  ) {$9$};
\node (10) at  (  {\c+\r*sin(10*\t-.5*\t)},    {\c+\r*cos(10*\t-.5*\t)}  ) {$10$};
\node (11) at (  {\c+\r*sin(11*\t-.5*\t)},    {\c+\r*cos(11*\t-.5*\t)}  )  {$11$};
\node (12) at  (  {\c+\r*sin(12*\t-.5*\t)},    {\c+\r*cos(12*\t-.5*\t)}  ) {$12$};
\node (13) at   (  {\c+\r*sin(13*\t-.5*\t)},    {\c+\r*cos(13*\t-.5*\t)}  ) {$13$};
\node (14) at  (  {\c+\r*sin(14*\t-.5*\t)},    {\c+\r*cos(14*\t-.5*\t)}  ) {$14$};
\draw   (1)--(2)--(3)--(4)--(5)--(6)--(7)--(8)--(9)--(10)--(11);
\draw  (11)--(12)--(13)--(14)--(1) (1)--(6)  (3)--(8)  (5)--(10);
\draw (7)--(12)  (9)--(14)  (11)--(2)  (13)--(4);
\end{tikzpicture} 
\caption {The Heawood graph}
\label{Hea}
\end{center}
\end{figure}
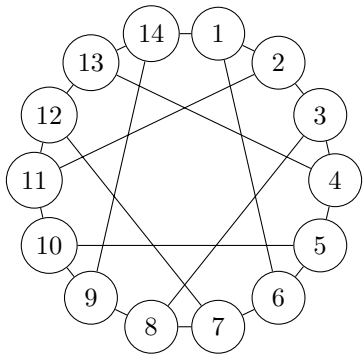

We shall now present a construction that assigns a split at a white vertex for every bi-transitive locally imprimitive graph of valence $\{4,k\}$.

\begin{construction}
\label{con:split}
The input of this construction is
a simple $G$-bi-transitive $2$-colored graph $\Gamma$ with black vertices of valency $k$ and white vertices of valency $4$,
and a system of imprimitivity $\Delta(v)$ for $G_v^{\Gamma(v)}$ for some white vertex $v$. The output is a split $\Delta = \Pairs(\Gamma, G, \Delta(v))$ at white vertices of $\Gamma$.

Observe first that since $G_v^{\Gamma(v)}$ is imprimitive, the permutation group $G_v^{\Gamma(v)}$ 
is permutation isomorphic to one of the groups $\ZZ_2^2$, $\ZZ_4$ or $\D_4$ in their natural transitive actions on four points. 
Observe that, in all three cases, there exists a normal subgroup $N$ of $G_v$ such that $\Delta(v) = \{e^N : e\in \Gamma(v)\}$.
For $g\in G$, let $\Delta(v)^g$ denote the set $\{D^g : D \in \Delta(v)\}$.

Since $G$ is transitive on the set of all white vertices, we can choose $g_w\in G$ for each white vertex $w$ such that $v^{g_w} = w$. 
Let $\Delta(w) = \Delta(v)^{g_w}$ and note that, since $N$ is normal in $G_v$, $\Delta(w)$ does not depend on the choice of $g_w$. 

We say that two edges of $\Gamma$ are $\Delta$-related if they have a common white endvertex $v$ and they are in the same block of $\Delta(v)$. It is not hard to see that $\Delta$ is a split at white vertices of $\Gamma$. \EOC
\end{construction}

\begin{theorem}
\label{theo:equivalent}
Let $\Gamma$ be a simple $G$-bi-transitive $2$-colored graph. Suppose that black vertices of $\Gamma$ have valency $k$ while white vertices have valency $4$, and that there exists a white vertex $v$ of $\Gamma$ such that $G_v^{\Gamma(v)}$ is imprimitive. Let $\Delta(v)$ be an imprimitivity system of $G_v^{\Gamma(v)}$.
Then $\Gamma\cong\BGCG(X,\Mate)$ where $X$ is a $k$-valent (not necessarily simple) 
$G$-dart-transitive graph and $\Mate$ is a $G$-invariant separating relation for $X$ such that each $\Mate$-equivalence class has cardinality $2$.
In fact, $X$ and $\Mate$ can be chosen to be $\Dis(\Gamma,\Delta)$ and $\Mat(\Gamma,\Delta)$, respectively, where $\Delta = \Pairs(\Gamma, G, \Delta(v))$.
\end{theorem}

\begin{proof}
Let $\Delta = \Pairs(\Gamma, G, \Delta(v))$,
let $X = \Dis(\Gamma,\Delta)$ and let $\Mate = \Mat(\Gamma,\Delta)$. Then $X$ is a $k$-valent graph and $\Mate$ is a separating relation for $X$. Since $\Gamma$ is $4$-valent, the $\Mate$-equivalence classes have cardinality 2. By Lemma~\ref{lem:inverse}, we have  $\BGCG(X,\Mate) \cong \Gamma$. By Lemma~\ref{lemma:yo}, there is a natural action of $G$ as an dart-transitive group of symmetry of $X$ and $\Mate$ is $G$-invariant.
\end{proof}

\section{The tetravalent case}

 For the rest of this paper, we consider the case in which $\Gamma$ is tetravalent.   In order to understand the BGCG construction, we first consider the dissection.  Let $\Gamma$ be a simple tetravalent $G$-bi-transitive $2$-colored graph  and let 
$\Delta(v)$ be  a block of imprimitivity for $G_v^{\Gamma(v)}$, for some white vertex $v$.
Let $\Delta = \Pairs(\Gamma, G, \Delta(v))$, let
$X = \Dis(\Gamma,\Delta)$ and let $\Mate = \Mat(\Gamma,\Delta)$.
Note that, in view of Theorem~\ref{theo:equivalent}, $\Gamma$ can be reconstructed as $\BGCG(X,\Mate)$.

Recall also that $X$ might be disconnected. Since $X$ is dart-transitive, connected components of $X$ are all isomorphic to
a fixed dart-transitive graph, $B$, called the {\it base graph} of $\Gamma$  (with respect to $G$ and $\Delta(v)$); hence $X\cong kB$ for some positive integer $k$.

The {\it connection graph} $C$ of $\Gamma$ (also with respect to $G$ and $\Delta(v)$), is the graph whose vertices are the connected components of $X$
and two such components $B_1$, $B_2$ are adjacent in $C$ whenever an edge from $B_1$ is $\Mate$-related to some edge from $B_2$.
Note that the connection graph $C$ is simple by definition, even when component $B_1$  has many edges $\Mate$-related to edges in $B_2$ .

Let us emphasize that the base graph and the connection graph really do depend on the choice of the system of imprimitivity $\Delta(v)$ and not just on $\Gamma$ and $G$.
 For example, if $G_v^{\Gamma(v)} \cong \ZZ_2^2$, then any partition of $\Gamma(v)$ into blocks of size $2$
is a system of imprimitivity for $G_v^{\Gamma(v)}$, and different partitions might give rise to different splits $\Delta$ and thus to different
graphs $X$. This phenomenon can be observed  in the examples considering the wreath graph $\Gamma=\W(6,2)$  in section \ref{DiS+BGCG}.

Now, we consider the BGCG {\it construction}.  That is, given a tetravalent, dart-transitive graph $B$ and a dart-transitive graph $C$ of order $k$, we would like to find every separating relation $M$ on $kB$ such that $\BGCG(kB,M)$ is an edge-transitive graph with connection graph isomorphic to $C$. This is our long-term goal.

Examining the Census \cite{C4} shows many examples of graphs which can be constructed this way; or we could say that these graphs admit dissections.   We have seen in Examples
  \ref{ex:first} and \ref{ex:second} constructions in which the base graph $B$ is $K_5$ and the connection graphs are $K_1$ and $K_2$ respectively.     Continuing these examples, we see in \cite{C4} graphs constructed from $B = K_5$   and $ C = \CC_n$ (for $n \in \{3,\ldots,12\}), K_6, K_{5,5}, K_{11},  K_{5,5,5}$  and others.  Further, there are constructions using base graphs $B = K_{2,2,2}$ (which is the skeleton of the Octahedron), $K_{4,4}$, the graph $\CC_3\Box\CC_3$ shown in Figure \ref{fig:kappa} below, and in fact almost all small tetravalent graphs. 

 For the rest of this paper, we will concentrate on the case when   $C=\K_1$ or $\K_2$, that is, those for which the  graph $X$ consists of one or two copies of $B$.

\subsection{The connection graph is $\K_1$}
\label{sec:K1}

The simplest case arises when the connection graph of $\Gamma$ (with respect to some $G$ and $\Delta(v)$) is $\K_1$.
Then $\Gamma$ is isomorphic to $\BGCG(X,\Mate)$ for some {\it connected} tetravalent $G$-dart-transitive graph $X$ (which, being connected, equals
the corresponding base graph $B$) and
a $G$-invariant separating relation $\Mate$ for $X=B$ such that each $\Mate$-equivalence class has size $2$. We thus introduce the following terminology.

\begin{definition}
A {\it dart-transitive pairing} of the edges of a tetravalent graph $B$ is a separating relation $\Pp$ on $\E(B)$ with equivalence classes, called {\it pairs}, of size $2$, such that the group of all symmetries of $B$ that preserve $\Pp$ (this group will be denoted by $\Aut(B,\Pp)$) is transitive on darts of $B$.
\end{definition}

Note that when $B$ is simple, then any dart-transitive relation with sets of size 2 is a separating relation, and so is a dart-transitive pairing.  We summarize the discussion above in the following theorem:
\begin{theorem}
Let $B$ be a tetravalent graph and let $\Pp$ be a separating relation for $B$. Then $\Pp$ is a dart-transitive pairing if and only if $\BGCG(B, \Pp)$ is edge-transitive.
\end{theorem}

To illustrate the concept of a dart-transitive pairing, we return to Example~\ref{ex:first}.

\begin{example}\label{example:later}
Let $B$ be the complete graph on $\ZZ_5$. Assign the color $i+j \in \ZZ_5$ to the edge  $\{i, j\}$ of $B$. Note that each color is assigned to exactly two edges.  Let $\Pp$ be the equivalence relation on $\E(B)$ of having the same color. It is not too hard to see that $\Aut(B,\Pp)$ contains the affine functions over $\ZZ_5$. In particular, $\Aut(B,\Pp)$ is transitive on darts, and thus $\Pp$ is a dart-transitive pairing of $B$. 

We can now construct the graph $\BGCG(B,\Pp)$, which happens to be isomorphic to the circulant graph $\CC_{10}(1, 3)$.  In fact, this example is simply the inverse of the dissection considered in Example~\ref{ex:first}.
\end{example}
\begin{example}\label{ex:c3c3}

Let $B=\CC_3\Box\CC_3$, that is, $B$ is the cartesian product of two $3$-cycles. The graph $B$  can be visualized as the skeleton of the map $\{4, 4\}_{3, 0}$, as shown in Figure \ref{fig:kappa}.

\begin{figure}[hhh]
\centering
\includegraphics[height=40mm]{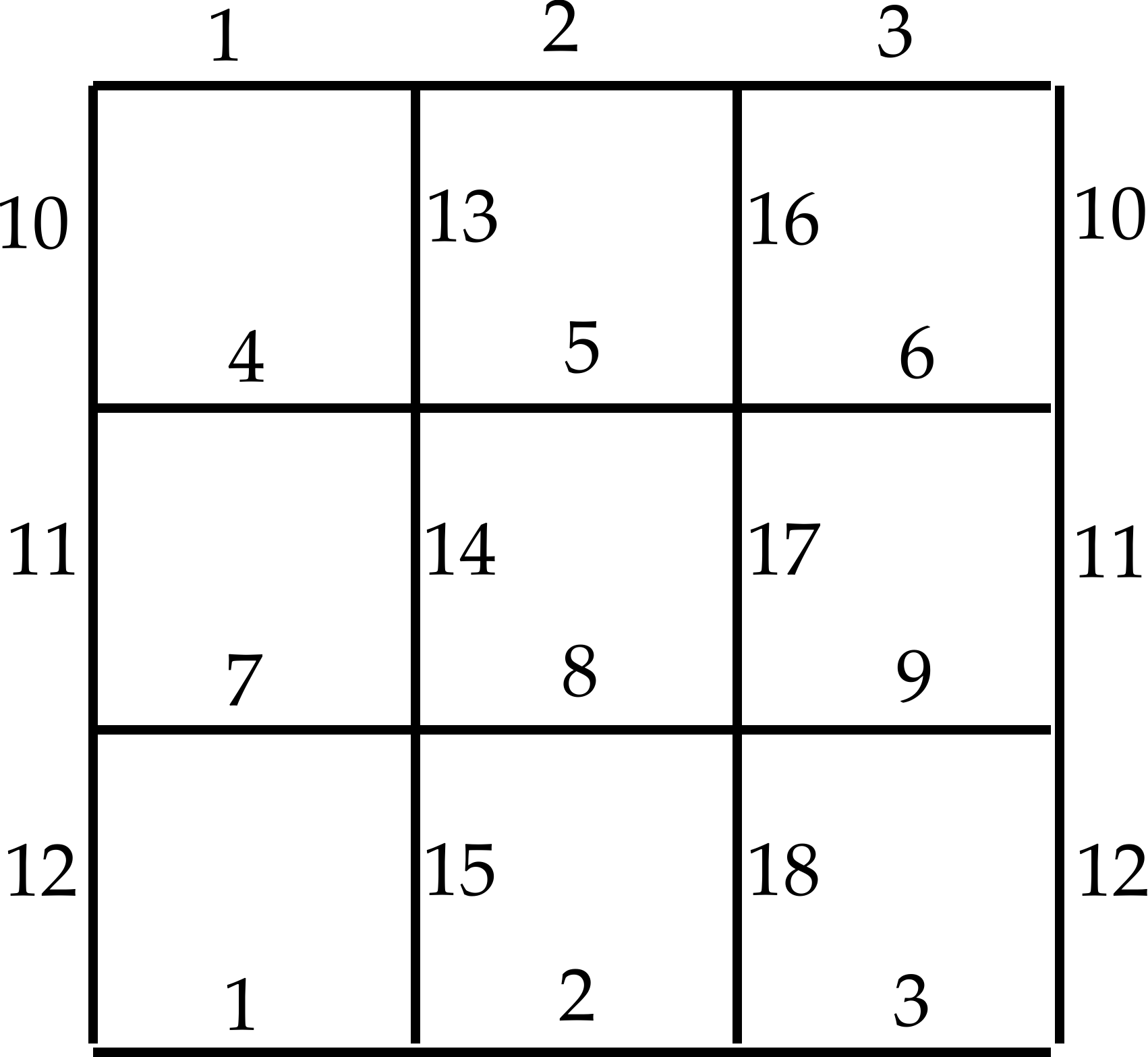}
\caption{The map $\{4, 4\}_{3, 0}$ 
}\label{fig:kappa}
\end{figure}

The group $H$ of orientation-preserving symmetries of the map is generated by the rotation about the central face $$R = (1~12~3~10)(4~15~9~16)(5~14~8~17)(6~13~7~18)(2~11)$$
 and  the rotation about the lower right corner of the central face $$ S = (1~13~4~10)(2~14~6~12)(3~15~5~11)(7~16)(8~17~9~18).$$ 

The group $H$ is transitive on the darts of the graph.  Consider this pairing: $$\Pp=\{\{1, 17\}, \{2, 11\}, \{3, 14\}, \{4, 18\}, \{5, 12\}, \{6, 15\}, \{7, 16\}, \{8,10\}, \{9, 13\}\}.$$  For each edge $e$, the stabilizer $H_e$ fixes exactly one other edge $e'$; $\Pp$ is the collection of all pairs of the form $\{e, e'\}$. Then $\Pp$ is preserved by $H$ and thus is a dart-transitive pairing of $B$.   Therefore, $\BGCG(B, \Pp)$ is a bi-transitive graph on 18 vertices; it is the skeleton of the toroidal map $\{4, 4\}_{3, 3}$.

\end{example}

Because these  pairings are so important, a natural question now arises:
\begin{question}
\label{q:cc}
Given a connected tetravalent dart-transitive graph $B$, how can one efficiently find all dart-transitive pairings of $B$, up to conjugacy in $\Aut(B)$?
\end{question}

In general, all dart-transitive pairings of a given dart-transitive graph $B$ can be obtained by using the following approach:
\begin{enumerate}
\item
Find all minimal dart-transitive subgroup $G\le \Aut(B)$ (up to conjugacy in $\Aut(B)$);
\item
For each minimal dart-transitive subgroup $G$, find all systems of imprimitivity for the action of $G$ on the edges of $B$ with blocks of size $2$;
\item
Reduce the set of imprimitivity systems modulo conjugacy in $\Aut(B)$.
\end{enumerate}

For groups in which the stabilizer of a dart has ``small'' order (less than 50, say), this approach is efficient.   Applying it to graphs of small size and with small vertex-stabilizer yields a variety of dart-transitive pairings; see examples in  Table~\ref{Table}, near the end of this paper.

The difficulty in implementing the algorithm might occur in step 1; namely, if $\Aut(B)$ has large order, then finding dart-transitive subgroups of $\Aut(B)$ might take an unreasonably long time. 
We first address Question~\ref{q:cc} in some special cases. Suppose  that the graph $B$ is not simple. Then it is not too hard to see that $B$ is either the {\it bouquet}, that is, a vertex with two loops attached to it, or the dipole with $4$ parallel edges, or a doubled cycle. The bouquet and the dipole  cannot have  separating relations, and so neither has a  dart-transitive pairing.  A doubled cycle has a dart-transitive pairing if and only if its length is even, in which case it has, up to isomorphism, a unique one (pairing diametrically opposed edges). Applying the BGCG construction to this pairing yields a wreath graph. 

Fortunately, connected simple tetravalent dart-transitive graphs with large dart-stabilizers are much better understood than they used to be.  In~\cite{PSV4valent}, it was shown that  there exists a sublinear function $f$ such that, apart from a certain exceptional family, the dart-stabilizer of a simple tetravalent graph of order $n$ has order at most $f(n)$. The exceptional family of graphs was first described in~\cite{PX}, and the ones with the largest vertex-stabilizers are also the simplest members: the wreath graphs, $\W(n, 2)$.   We  consider them extensively in  the next section.

\section{Wreath Graphs and their Dart-transitive pairings}\label{ex:wreath}
In this section, we consider the troublesome family $\W(n,2)$ of graphs, determine all of their dart-transitive pairings and the graphs  resulting from applying the BGCG construction to them.

Consider the graph $\Gamma = \W(n,2)$. If $n=4$, then $\W(n,2)$ is isomorphic to the complete bipartite graph $\K_{4,4}$, and thus its symmetry group is isomorphic to $\Sym(4)\wr \Sym(2)$.  For all other values of $n$, the collection of pairs of the form $\{(i,0),(i,1)\}, i\in \ZZ_n$, forms a system of imprimitivity for $\Aut(\W(n,2))$. In particular, $\Aut(\W(n,2))$ is generated by symmetries $\rho, \mu$, and $\tau_i$ for $i\in \ZZ_n$, where $(i,j)^\rho  = (i+1, j)$, $(i,j)^\mu = (-i, j)$ for all $i\in \ZZ_n$, $j\in \ZZ_2$, and where $\tau_i$ interchanges $(i,0)$ and $(i,1)$ leaving all other vertices fixed. Thus,  the symmetry group $\Aut(W(n,2)$ is isomorphic to $\ZZ_2\wr\D_n$ and so its order is $n\cdot 2^{n+1}$.

We will find it useful to have labels for the edges of $\W(n,2)$ themselves.  Let $a_i = \{(i, 0), (i+1, 0)\}$, $b_i = \{(i, 0), (i+1, 1)\}$, $c_i = \{(i, 1), (i+1, 1)\}$, and $d_i = \{(i, 1), (i+1, 0)\}$, as in Figure~\ref{fig:handy}.  Finally, we let $C_i = \{a_i, b_i, c_i, d_i\}$.

\begin{figure}[hhh]
\begin{center}
\includegraphics[height=22mm]{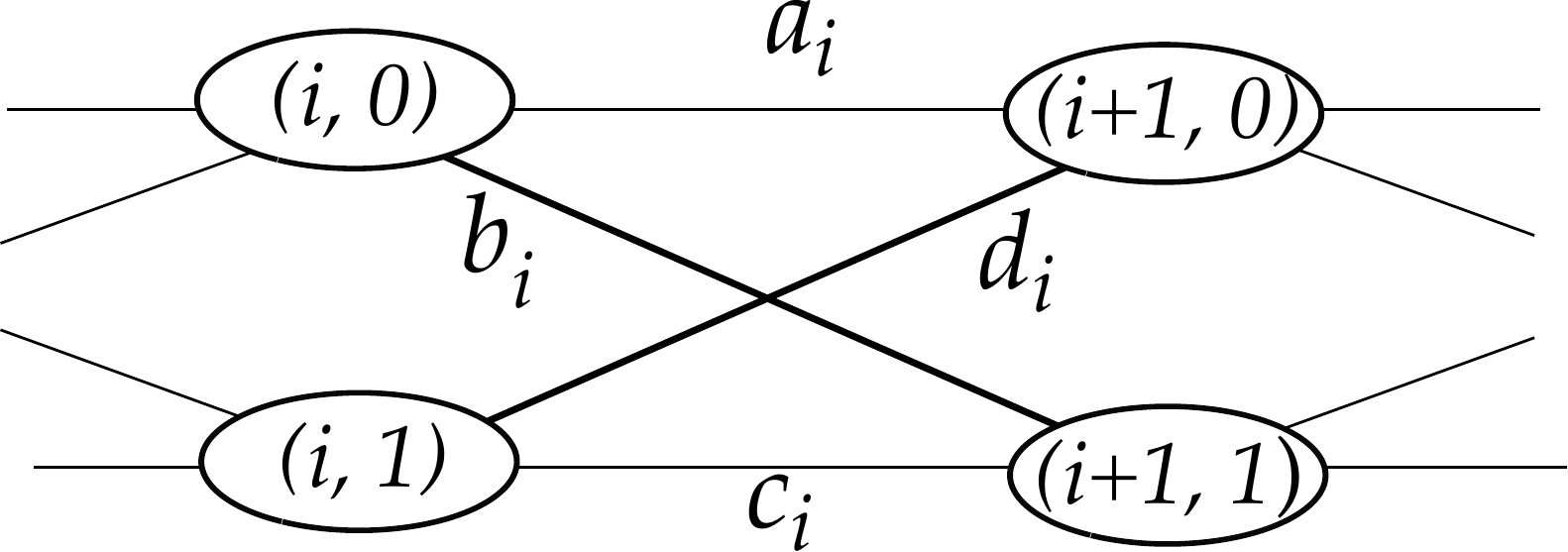}
\caption{Labels for edges of $\W(n,2)$}
\label{fig:handy}
\end{center}
\end{figure}

First, if $\Pp$ is a pairing on $\Gamma$ and if  $G$, the group  of symmetries preserving the pairing, is to be dart-transitive then $G$ must contain:
\begin{enumerate}
\item A symmetry which sends each pair $\{(i,0), (i, 1)\}$ to $\{(i+1,0), (i+1, 1)\}$.  Without loss of generality, we can assume this symmetry $\rho^*$ is $\rho$ or $\rho' = \rho\tau_0$.
\item A symmetry which sends each pair $\{(i,0), (i, 1)\}$ to $\{(-i,0), (-i, 1)\}$.   This $\mu^*$ must be $\mu$ times some product of the $\tau_i$'s.
\item A symmetry $\tau^*$ which fixes $(0,0)$ and interchanges $(1,0)$ with $(1, 1)$.  Then $\tau^*$ must itself be a product of the $\tau_i$'s.

\end{enumerate}

In searching for dart-transitive pairings, we first note the obvious one, labelled `$\Pp_0$' in Table \ref{Table2}, in which each $a_i$ is paired with $c_i$, and each $b_i$ with $d_i$.  This is invariant under all of $\Aut(\Gamma)$.   

Any other pairing must pair each edge in $C_i$ with some edge in $C_{i+m}$ where $n$ is even and equal to $2m$.  We can indicate the 4 pairs within $C_i\cup C_{i+m}$ using labels $[1,i], [2, i], [3,i], [4,i]$.  We can assume without any loss of generality that the edges $a_0, b_0, c_0, d_0$ belong to pairs $[1,0], [2, 0], [3,0], [4,0]$ respectively, and that for $i = 0, 1, \dots, m-2$, $[c, i]\rho^* = [c, i+1]$ for all $c\in \{1, 2, 3 ,4\}$.   Then for some $\sigma\in D_4$, we must have $[c, m-1]\rho^* = [c\sigma, 0]$ for all $c\in \{1, 2, 3 ,4\}$, and this $\sigma$ (together with the choice of $\rho^*$) completely determines the pairing.

\begin{table}[hhh]
\begin{center}
\begin{tabular}{|c|c|c|c|c|c|c|l|l|}
\hline
$\Pp_i$  &  $\sigma$ &$m$ &  $\rho^*$  &  $\mu^*$  &  $\tau^*$  &$|G|$&  $K_1$  &  $K_2$\\
\hline
 
$\Pp_0$&-& -&$\rho$  &  $\mu$  &  $\tau_1$  &  $2n2^n$  &  $\W(2n,2)$  &  $\SDD(\W(n,2))$\\

\hline
$\Pp_1$&Id& any&$\rho$  &  $\mu$  &  $\tau_1\tau_{m+1}$  &  $4m2^m$  &  $\SDD(\W(m,2)$  &  $\SDD(\W(n,2))$\\
\hline
$\Pp_2$&$(2  4)$&odd&$\rho$  &  $\mu$  &  $\tau_1\tau_3\tau_5\dots\tau_{n-1}$  &  $16m$  &  $\{4, 4\}_{[m,4]}$  &  $\{4, 4\}_{[n,4]}$  \\
\hline
$\Pp_3$&$(1   2)(3  4)$&any&$\rho$  &  $\mu\tau_1\tau_2\tau_3\dots\tau_{m}$  &  $\tau_1\tau_{m+1}$  &   $4m2^m$  &  $\PX(n,2)$  &  $\PX(2n,2)$  \\
\hline
$\Pp_4$&$(1   2  3  4)$&odd&$\rho\tau_0$  &  $\mu\tau_2\tau_4\tau_6\dots\tau_{n-2}$  &  $\tau_1\tau_3\tau_5\dots\tau_{n-1}$  &   $16m$  &  $\C_{8m}(1,2m+1)$  &  $\{4, 4\}_{[n,4]}$ \\
\hline
$\Pp_5$&$(4  3  2  1)$&odd&$\rho\tau_0$  &  $\mu\tau_2\tau_4\tau_6\dots\tau_{n-2}$  &  $\tau_1\tau_3\tau_5\dots\tau_{n-1}$  &   $16m$  &  $\C_{8m}(1,2m-1)$  &  $\{4, 4\}_{[n,4]}$ \\

\hline
$\Pp_6$&$(1  3)(2  4)$& any&$\rho$  &  $\mu$  &  $\tau_1\tau_{m+1}$  &  $4m2^m$  &  $\SDD(\W(m,2)$  &  $\SDD(\W(n,2))$\\
\hline
$\Pp_7$&$(1  3)$&odd&$\rho$  &  $\mu$  &  $\tau_1\tau_3\tau_5\dots\tau_{n-1}$  &  $16m$  &  $\{4, 4\}_{[m,4]}$  &  $\{4, 4\}_{[n,4]}$  \\
\hline
$\Pp_8$&$(1   4)(2  3)$&any&$\rho$  &  $\mu\tau_1\tau_2\tau_3\dots\tau_{m}$  &  $\tau_1\tau_{m+1}$  &   $4m2^m$  &  $\PX(n,2)$  &  $\PX(2n,2)$  \\
\hline
\end{tabular}
\caption{All dart-transitive pairings in wreath graphs}\label{Table2}
\end{center}
\end{table}

All eight of these pairings are dart-transitive; the requisite symmetries $\rho^*, \mu^*, \tau^*$,  as well as the size of the group $G$ that they generate,  are given in Table \ref{Table2}.   Not all of them are distinct: pairings 6, 7, 8 are isomorphic to pairings 1, 2, 3, respectively.   In each case, the isomorphism is given by the symmetry $\tau_1\tau_2\tau_3\dots\tau_{m}$.

The  column labelled `$K_1$' gives the result of the construction $\BGCG(W(n, 2), K_1, \Pp_i)$.  We provide some notation for these entries: \begin{enumerate}
\item $\PX(n,k)$,  introduced in \cite{PX}, has vertex set $\ZZ_n\times\ZZ_2^k$, and its edges are all $\{(i, jx), (i+1, xj')$where $j$ and $j'$ are bits and $x$ is a bitstring of length $k-1$.  $\PX(n,1)$ is isomorphic to the wreath graph $\W(n,2)$, and $\PX(n,2)$ is isomorphic to the rose window graph $\R_{2n}(n+2, n+1)$; see \cite{Rose}.
\item The graph $\{4, 4\}_{[b,c]}$ is the skeleton of the map formed from the tessellation $\{4, 4\}$ of the plane into squares meeting 4 at each vertex by factoring out the translation group $T_1$ generated by translations$(b, b)$ and $(-c, c)$.
\item The graph $\{4, 4\}_{<b,c>}$is similarly formed, using the group $T_2$ generated by $(b, c)$ and $(c, b)$
\end{enumerate}
  In Table \ref{Table2}, we    avoid some separation into cases by using these facts:\begin{enumerate}
\item If $m \equiv 1$ (mod 4), then $\{4, 4\}_{[4, m]} \cong  C_{8m}(1, 2m-1)$ and $\{4, 4\}_{<m+2, m-2>} \cong C_{8m}(1, 2m+1)$
\item If $m \equiv 3$ (mod 4), then $\{4, 4\}_{[4, m]} \cong  C_{8m}(1, 2m+1)$ and $\{4, 4\}_{<m+2, m-2>} \cong C_{8m}(1, 2m-1)$
\end{enumerate}
 The final column of the table will be referred to in the next section.

\section{The connection graph is $\K_2$}
\label{sec:K2}

Suppose now that the connection graph of $\Gamma$ (with respect to some $G$ and $\Delta(v)$) is $\K_2$.
If $B$ is the corresponding base graph, then $\Gamma$ is isomorphic to $\BGCG(X,\Mate)$ 
for some $G$-invariant separating relation $\Mate$ on a graph $X$ having exactly two components, each isomorphic to $B$.

Let $B_0$ and $B_1$ be the components of $X$ and let $\phi_0:B \rightarrow B_0$ and $\phi_1:B \rightarrow B_1$ be isomorphisms. Given $x$ an edge or vertex of $B$ and $i\in\{0,1\}$, we give the label $(x, i)$ to $x^{\phi_i}$.

Since the relation $\Mate$ on $\E(X)$ is invariant under the dart-transitive group $G$ and since $\{\E(B_0),\E(B_1)\}$
is clearly a $G$-invariant partition of $\E(X)$, it follows that either each equivalence class of $\Mate$ is
contained in one of $\E(B_0)$ and $\E(B_1)$, or each equivalence class of $\Mate$ intersects both $\E(B_0)$ and $\E(B_1)$.
In the former case, the corresponding graph $\BGCG(X,\Mate) \cong \Gamma$ is disconnected,
which we have ruled out. Hence the latter possibility occurs. This allows us to define a permutation $\kappa$ on the set $\E(B)$
such that the equivalence classes of $M$ are of the form $\{(e,0), (e^\kappa,1)\}$.

To be explicit, if we denote by $e$ the white vertex $\{(e, 0), (e^\kappa, 1)\}$, then  $(u,0)$ is adjacent to $e$ if $u\in e$ and $(v,1)$ is adjacent to $e$ if $v \in e^\kappa$.

Conversely, given the graph $B$ and  a permutation $\kappa$ on $\E(B)$, let $2B$ be the disjoint union of two copies of $B$ indexed by $\{0,1\}$.  One can then define a relation $M$ on $\E(2B)$ by letting the edge $(e, 0)$ be $M$-related to
$(e^\kappa, 1)$. This relation $M$ is a separating relation on $2B$. We then define $\BGCG(B, \K_2, \kappa)$ to be $\BGCG(2B, \Mate)$.

\begin{example} If $\kappa$ is the identity permutation then $\BGCG(B,\K_2,\kappa)\cong\SDD(B)$.
\end{example}

\begin{lemma}\label{newlabel}
 Let $B$ be a connected tetravalent simple graph. Let $A=\Aut(B)$ and let $\kappa$ and $\kappa'$ be permutations on $\E(B)$ such that, when $A$ is viewed as a permutation group on $\E(B)$, we have $\kappa' \in A\kappa A$. Then  $\BGCG(B,\K_2,\kappa)\cong \BGCG(B,\K_2,\kappa')$.
\end{lemma}
\begin{proof} 
Let $\Gamma=\BGCG(B,\K_2,\kappa)$ and $\Gamma'=\BGCG(B,\K_2,\kappa')$ and let $\gamma,\gamma'\in A$ such that, as permutations of $\E(B)$, we have $\kappa' =\gamma^{-1}\kappa\gamma'$, so that $\kappa\gamma'= \gamma\kappa'$.

Let $\phi$ be the mapping from $\V(\Gamma)$ to $\V(\Gamma')$ defined by $(u,0)^\phi=(u^{\gamma},0)$, $(v,1)^\phi=(v^{\gamma'},1)$, for $v\in\V(B)$ and $e^{\phi}=e^{\gamma}$ for $e\in\E(B)$.  Suppose that $(u, 0)$ and $(v, 1)$ are adjacent to $e$.   Since $v \in e^{\kappa}$, it follows that  $v^{\gamma'} \in (e^{\kappa})^{\gamma'} = e^{\gamma\kappa'} = (e^{\gamma})^{\kappa'} $.  Thus $\phi$ is an isomorphism.
\end{proof}

We will say that $\kappa$ is a {\it push} of $B$ provided that the corresponding partition of edges given by $\Pp = \{\{(e, 0), (e^\kappa, 1)\} \mid e \in \E(B)\}$ is a dart-transitive pairing of the graph $2B$, in other words, if $\Aut(2B, \Pp)$ is transitive on darts of $2B$.

The idea of a dart-transitive pairing, which is essential in the case when $X$ is connected, is also useful here in the 2-component case: given a dart-transitive pairing $\Pp$ of $B$, we construct a push $\kappa_\Pp$ of $2B$
by letting $e^{\kappa_\Pp}$ be the other edge of $B$ in the same $\Pp$-class as $e$. (Note that such a push $\kappa_\Pp$ is an involution without fixed points.)
It is not hard to see that, if $B$ is connected, then so is $\BGCG(B,\K_2,\kappa_\Pp)$. 

The final column of Table \ref{Table2} shows the result of  $\BGCG(W(n,2),\K_2, \kappa_{\Pp_i})$ for each dart-transitive pairing $\Pp_i$.

A further curious fact:  for even values of $m$, even though the pairings $\Pp_4$ and $\Pp_5$ are not dart-transitive, the graphs $\BGCG(W(2m,2), K_2, \Pp_4)$ and $\BGCG(W(2m,2), K_2, \Pp_5)$ are edge-transitive; both are isomorphic to $\{4, 4\}_{<n+2, n-2>}$.

In Section~\ref{sec:last}, we will consider graphs obtained via dart-transitive pairings of edges in small graphs other than wreath graphs. Not every push comes from a dart-transitive pairing, though, and so we need to consider pushes of $X$ more generally.

We now prove two results which determine what conditions on $B$ and a permutation $\kappa$ of $\E(B)$
will insure that $\kappa$ is a push.

\begin{theorem}
\label{th:K2}
Let $B$ be a connected tetravalent dart-transitive simple graph and let $\kappa$ be a permutation of $\E(B)$ satisfying these two conditions:
\begin{enumerate}
\item there exist dart-transitive subgroups $G_0$ and $G_1$ of $\Aut(B)$ such that, when viewed as permutation groups on $\E(B)$, we have $(G_0)^\kappa=G_1$;
\item there exist $\tau_0,\tau_1\in\Aut(B)$ such that, when viewed as permutations on $\E(B)$, we have $\tau_0=\kappa\tau_1\kappa$.
\end{enumerate}
Then $\kappa$ is a push  of $B$.  Conversely, if $\kappa$ is a push of $B$ then it satisfies these conditions.
\end{theorem}
\begin{proof}

We first show that if the conditions hold, then the pairing of edges of $2B$ induced by $\kappa$ (that is, the collection $\Pp$ of all pairs of the form $\{(0, e), (1, e^\kappa)\}$)
is invariant under some dart-transitive group of symmetries of $2B$.

In a simple graph, the action of the symmetry group on the vertices is faithful. If the graph is also connected and has at least two edges, then the action on edges is also faithful. This implies that, given $g\in G_0$, there exists a unique element $g^*\in G_1$ such that, when viewed as permutations of $\E(B)$, we have $g^*=g^\kappa=\kappa^{-1} g\kappa$.

For every $g \in G_0$, let $\bar{g}$ act on $2B$ by $(x,0)^{\bar{g}}=(x^g,0)$ and $(x,1)^{\bar{g}}=(x^{g^*},1)$, for every $x\in\V(X)\cup\E(X)$. Let $\bar{G}=\{\bar{g}\mid g\in G_0\}$. Clearly, $\bar{G}$ is closed under composition and is thus a subgroup of $\Aut(2B)$ preserving the two connected components of $2B$.  Let $e\in\E(B)$. Then $(e,0)$ is in the same pair as $(e^\kappa,1)$, but also, $(e,0)^{\bar{g}}=(e^g,0)$ is in the same pair as $(e^{g\kappa},1)=(e^{\kappa},1)^{\bar{g}}$. This shows that $\bar{G}$ preserves $\Pp$. Finally, note that, since $G_0$ and $G_1$ are both dart-transitive on $B$, the group $\bar{G}$ acts dart-transitively on each component. 

Let $\beta$ be the permutation of $\V(2B)\cup\E(2B)$ defined by $(x, 0)^\beta = ( x^{\tau_0},1 )$ and $(x,1)^\beta = (x^{\tau_1}, 0)$ for every $x\in\V(B)\cup\E(B)$. It is easy to check that $\beta$ is a  symmetry of $2B$ preserving $\Pp$ that exchanges the two components. It follows that the group generated by $\bar{G}$ and $\beta$ preserves $\Pp$ and is dart-transitive.   Thus $\kappa$ is a push   of $B$.

Conversely, we wish to show that any push $\kappa$ of $B$ satisfies the two conditions of the theorem. Let $M$ be the separating relation on $2B$ corresponding to $\kappa$. Since $\kappa$ is a push,  $M$ is invariant under some dart-transitive group $G$ of $2B$.  Let $G^+$ be the subgroup of $G$ which preserves $B_0$ (and so also $B_1$).   For each $g \in G^+$, define $g_0, g_1$ by $(x,0)^g = (x^{g_0}, 0), (x,1)^g = (x^{g_1}, 1)$.    Let $G_0 = \{g_0| g\in G^+\}$ and $G_1 = \{g_1| g\in G^+\}$.   Since $G$ preserves $M$, we have
  $\{(e,0), (e^\kappa, 1)\}^g =  \{(e^{g_0},0), (e^{\kappa g_1},1)\} \in M$, 
and hence, in the action on edges, $g_0\kappa = \kappa g_1$, and thus $G_0^\kappa = G_1$.  

	Now let $\beta$ be any element of $G$ which exchanges $B_0$ and $B_1$, and define $\tau_0, \tau_1$ by  $(x, 0)^\beta = (x^{\tau_0},1)$ and $(x, 1)^\beta = (x^{\tau_1},0)$ for every $x\in\V(B)\cup\E(B)$.   Since $G$ preserves $M$, each  $\{(e, 0), (e^\kappa,1)\}^\beta  =  \{(e^{\tau_0},1), (e^{\kappa\tau_1}, 0)\} \in M$, and so, in the action on edges, $\tau_0 = \kappa \tau_1\kappa$.
\end{proof}

\begin{corollary}
\label{th:K2c}
Let $B$ be a connected tetravalent dart-transitive simple graph, let $A=\Aut(B)$,
and let $\kappa$ be a permutation of $\E(B)$ satisfying these two properties:
\begin{enumerate}
\item[(1)] there exists some dart-transitive subgroup $H$ of $A$ such that, when viewed as a permutation group on $\E(B)$, $H$  is normalized by $\kappa$; and
\item[(2)] there exists $\tau\in A$, such that $\kappa^2 = \tau$, when $\tau$ is viewed as a permutation of $\E(B)$.
\end{enumerate}
Then $\kappa$ is a push  of $B$ and $\BGCG(B, K_2, \kappa)$ is bi-transitive.  In a converse direction, let $X=2B$, suppose that $\kappa'$ is a permuation of  $\E(B)$ and let $\Gamma =\BGCG(X, \kappa')$. If $\Gamma$ is connected and bi-transitive then the permutation $\kappa'$ is in the double coset $A\kappa A$ for some $\kappa$  which  satisfies these two conditions.
\end{corollary}

\begin{proof}
 By applying Theorem~\ref{th:K2} with $\tau_1=\id_B$,  $\tau_0=\tau$, and $G_0=G_1=H$, we see that any $\kappa$ satisfying these two conditions is a push of $B$ and hence, by Theorem~\ref{theo:equivalent2}, $\BGCG(B, \K_2, \kappa)$ is bi-transitive.  

Conversely, let $\beta$ be any symmetry of $\Gamma$ which interchanges the two components of $X$.   If $\kappa$ is a push then the conditions of Theorem~\ref{th:K2} are satisfied.  Let $\phi_0$ be any isomorphism of $B$ onto component $B_0$ of $X$.  Then we can choose $\phi_1:B \rightarrow B_1$ to be $\phi_0\beta$.  Following the last paragraph of the proof of Theorem \ref{th:K2}, this choice forces $\tau_0$ to be the identity on $B$, and so $\tau_1^{-1} = \kappa^2$ in its action on edges.  Now  let $H = A\cap A^\kappa$.  It follows that $H$ is normalized by $\kappa$.  Since $G_1 \le A$ and $G_1 = G_0^\kappa \le A^\kappa$, $G_1$ must be a subgroup of $H$, and so H is transitive on darts of $B$.
\end{proof}

\subsection{Other pushes of wreath graphs}\label{sec:wreath}

Let us now provide examples of pushes  of the wreath graphs which do not arise from dart-transitive pairings of edges.

\begin{example}\label{ex:2}
Suppose $n$ is even  and let $B= \W(n,2)$. Let  $\rho, \mu$, $\tau_i$, $b_i$ and $d_i$  be as in Section~\ref{ex:wreath}. Let
  $$\kappa = (b_{0}~d_{0})(b_{1}~d_{1})\cdots(b_{n-1}~d_{n-1}),$$
  viewed as a permutation of $\E(B)$. Since $\kappa$ is an involution, it certainly satisfies (2) of Corollary~\ref{th:K2c}.  Let $\alpha = \tau_0\tau_2\cdots\tau_{n-2}$ and let $H=\langle\rho,\mu,\alpha\rangle$. Note that $\alpha^{\rho}=\tau_1\tau_3\cdots\tau_{n-1}\in H$ and thus $H$ is dart-transitive. It is easy to see that, as permutations of $\E(B)$, $\kappa$ commutes with both $\rho$ and $\mu$. With a little more effort, one can see that $\alpha^\kappa = \alpha^{\rho}\in H$ and thus $\kappa$ normalises $H$. By Corollary~\ref{th:K2c}, $\kappa$ is a push of $B$. 
\end{example}

\begin{example} Suppose $n$ is divisible by $4$ and write $n = 2m$, with $m$ even. Let $B= \W(n,2)$ and let $\rho, \mu$, $\tau_i$, $b_i$ and $d_i$ be as in Section~\ref{ex:wreath}. Let
 $$\kappa = \prod_{i = 0}^{m-1}(b_{i}~b_{i+m})(d_{i}~d_{i+m}),$$
   viewed as a permutation of $\E(B)$.  Since $\kappa$ is an involution, it certainly satisfies (2) of Corollary~\ref{th:K2c}. As in Example~\ref{ex:2}, let $\alpha = \tau_0\tau_2\cdots\tau_{n-2}$ and let $H=\langle\rho,\mu,\alpha\rangle$. Again, $H$ is dart-transitive and  $\kappa$ centralizes  $\langle\rho, \mu\rangle$. Computation shows that $\alpha^\kappa = \alpha\rho^m\in H$ and thus $\kappa$ normalises $H$. By Corollary~\ref{th:K2c}, $\kappa$ is a push of $B$.
\end{example}

\begin{example}
Suppose $n$ is odd and write $n = 2m+1$. Let $B= \W(n,2)$ and let $\rho, \mu$, $\tau_i$, $a_i$, $b_i$, $c_i$ and $d_i$ be as in Section~\ref{ex:wreath}. Let
 $$\kappa = (a_0~b_0)(c_0~d_0)\prod_{i = 1}^{m}(a_{2i}~c_{2i})(b_{2i}~d_{2i}),$$
    viewed as a permutation of $\E(B)$. Since $\kappa$ is an involution, it certainly satisfies (2) of Corollary~\ref{th:K2c}. Let $H=\langle \rho, \mu,\tau_{1}\rangle$. Clearly, $H$ is dart-transitive. We now show that $H$ is normalised by $\kappa$. Let $\tau= \tau_{2}\tau_{4}\cdots\tau_{2m}$.  The reader can check that $\rho^\kappa = \rho^\tau\in H$. Similarly, it is not hard to check that $\mu^\kappa = \mu\tau_0\tau_1\cdots\tau_{n-1}$. Finally, as a permutation of $\E(B)$, we have $\tau_1=(a_0~b_0)(c_0~d_0)(a_1~d_1)(c_1~b_1)$. In particular, $\tau_1$ commutes with $\kappa$  and thus $\kappa$ normalises $H$. By Corollary~\ref{th:K2c}, $\kappa$ is a push of $B$.  We offer without proof the claim that $\BGCG(B, \K_2, \kappa)$ is isomorphic to the rose window graph $\R_{4n}(2n+2, 2n+1)$.
\end{example}

\subsection{Non-involutory pushes}\label{sec:nonI}

 Note that all the pushes of wreath graphs given in Section~\ref{sec:wreath} are involutions and thus trivially satisfy condition (2) of Corollary~\ref{th:K2c}. We now give an example of a non-involutory push.

\begin{example}
Consider the graph $B$ of Example \ref{ex:c3c3}, shown in Figure \ref{fig:kappa}, and the permutation $$\kappa = (1~4~3~9)(5~7~8~6)(10~15~12~16)(13~17~18~14).$$ on its edges.  Computation shows that $R^\kappa = R^{-1}$ and $S^\kappa$ is rotation about the top center face in the figure.  Thus $\kappa$ normalizes  $H$, and $\kappa^2 = R^2$; hence $\kappa$ satisfies the conditions of Corollary \ref{th:K2c} and is thus a push for  $B$. This example shows that a push need not be an involution; in fact, every  element of the double coset of $\Aut(B)$ containing $\kappa$ has order 4 or 8.

\end{example}

\subsection{Non-wreath base graphs}\label{sec:last}

To show the variety possible, we have examined all the connected dart-transitive tetravalent simple graphs $B$ of order at most $16$ that are not isomorphic to a wreath graph and, for each of these, computed all their dart-transitive pairings up to conjugacy in $\Aut(\Gamma)$. For each  graph $B$ and dart-transitive pairing $\Pp$, we then constructed the graphs $\BGCG(B,\Pp)$ and $\BGCG(B, \K_2,\kappa_\Pp)$. The results are summarized in Table~\ref{Table}. (Instead of describing the pairings explicitly, we have simply numbered them. The notation for the graphs follows~\cite{C4}.)

\begin{table}[hhh]
\begin{center}
\begin{tabular}{|c|c|c|c|c|c|}
\hline
 $B$ & $\Pp$ &\multicolumn{1}{|c|}{$\BGCG(B,\Pp)$}&\multicolumn{1}{|c|}{$\BGCG(B, \K_2,\kappa_\Pp)$}\\
\hline
 
$K_5$ & 1 & $\CC_{10}(1,3)$ & $\R_{10}(4, 1)$\\

\hline
DW(3, 3) & 1 & DW(6, 3) & $\{4, 4\}_{6,0}$\\

\hline
$\CC_{10}(1,3)$ & 1 & $\SDD(\K_5)$ & SDD($\CC_{10}(1,3))$\\

\hline
$\R_6(1, 2)$ & 1  & SDD(Octahedron) & HC(F8)\\

\hline
$\R_6(1, 2)$ & 2  & $\R_{12}(8, 7)$ & HC(F8)\\

\hline
$\R_6(1, 2)$ & 3 & SDD(Octahedron)  & SDD($\R_6(5,4)$)\\

\hline

$\CC_{13}(1,5)$ & 1  & $\CC_{26}(1,5)$  & $\R_{26}(10, 1)$\\

\hline
L(Petersen) & 1  & PS(6,5;2)  & HC(F10)\\

\hline
$\CC_{15}(1,4)$ & 1  & $\CC_{30}(1,11)$  & $\R_{30}(22, 1)$\\

\hline
$\R_8(6, 5)$ & 1  & $\R_{16}(10, 9)$  & PL(SoP(4,4))\\

\hline
$\R_8(6, 5)$  & 2  & SDD($\K_{4,4}$)  & PL(SoP(4,4))\\

\hline
$\R_8(6, 5)$  & 3  & SDD($\K_{4,4}$)  & SDD($\R_8(6, 5)$)\\

\hline
$\R_8(6, 5)$  & 4   & MSY[4, 8, 3, 4]  &AMC[8,8, (3 6):(4 5)]\\

\hline
$\R_8(6, 5)$ & 5  & $\{4, 4\}_{4, 4}$  & AMC[8,8, (3 6):(4 5)]\\

\hline
\end{tabular}
\caption{Dart-transitive pairings in small non-wreath graphs}\label{Table}
\end{center}
\end{table}

\section{ Final questions}

\begin{enumerate}
\item Given a  connected tetravalent dart-transitive graph, we have already considered the problem of finding all of its dart-transitive pairings (see Question~\ref{q:cc} and the follow-up).  As we have seen in the discussion following Lemma~\ref{newlabel}, this also yields many of the pushes of the graph. How do we find the other pushes? Corollary \ref{th:K2c} appears to give an answer but its practical implementation appears to be difficult. Is there a way to define a 'canonical' representative $\kappa'$ of $A\kappa A$ which satisfies the conditions?
\item In subsection \ref{sec:wreath}, we exhibit a number of pushes of the wreath graph. Of course, the pushes arising from the dart-transitive pairings shown in section \ref{ex:wreath}   also exist.  Are those all?
\item  After $\K_1$ and $\K_2$, the next simplest connection graphs to consider are $k$-cycles, with $k\geq 3$.  How can one decide if a graph $B$ is suitable as a base graph for a $\BGCG$ construction with connection graph being a $k$-cycle?
\item Given a graph $B$ and a natural number $k$, how can we determine the dart-transitive pairings of $kB$?

\end{enumerate}


\begin{thebibliography}{99}

\bibitem{BonStel}
J.. van Bon, B.\ Stellmacher, 
On locally s-arc transitive graphs that are not of local characteristic $p$,
{\em J.\ Algebra} {\bf 528} (2019), 1--37.

\bibitem{CM} H.S.M.\ Coxeter and W.O.J.\ Moser, ``Generators and Relations for Discrete Groups'', Springer-Verlag (1972).

\bibitem{feng}
Y.-Q.\ Feng, Y.\ Wang,
Bipartite edge-transitive bi-p-metacirculants,
{\em Ars Math.\ Contemp.} {\bf 17} (2019) 591--615.


\bibitem{GLP} M.\  Giudici, C.\ H.\ Li, C.\ E.\ Praeger, Analysing finite locally $s$-arc transitive graphs,  {\em Trans.\ Amer.\ Math.\ Soc.} {\bf 356} 
 (2004),  291--317.
 
 \bibitem{jaj}
R.\ Jajcay, \v{S}. Miklavi\v{c}, P.\ \v{S}parl, G.\ Vasiljevi\'c, On certain edge-transitive bicirculants,
{\em Electron.\ J.\ Combin.} {\bf 26} (2019), Paper 2.6.

\bibitem{PSV4valent} P.~Poto\v{c}nik, P.~Spiga and G.~Verret, Bounding the order of the vertex-stabilizer in $3$-valent vertex-transitive and $4$-valent arc-transitive graphs, {\it J. Combin. Theory Ser.B}, \textbf{111} (2015) , 148--180.

\bibitem{C4} P.~Poto\v{c}nik and S.~Wilson, Census of  tetravalent edge-transitive graphs, 
\href{http://jan.ucc.nau.edu/~swilson/C4Site/Glossary.html}{http://jan.ucc.nau.edu/~swilson/C4FullSite/Glossary.html}, accessed February 2020.

\bibitem{recipe} P.~Poto\v{c}nik and S.~Wilson, Recipes for Edge-transitive Tetravalent Graphs, arXiv:1608.04158 [math.CO], August 2016.

\bibitem{PotWilg4} P.\ Poto\v{c}nik, S.\ Wilson,
  Tetravalent edge-transitive graphs of girth at most 4,
      {\it Journal of Combinatorial Theory Ser.~B},
       {\bf 97} (2007), 217--236. 

\bibitem{PX}  C. Praeger and M.-Y. Xu, A characterization of a class of symmetric graphs of twice prime valency, {\it European J. Combin.}, \textbf{10} (1989) , 91--102.

\bibitem{pablo}
P.\ Spiga,  An application of the local $C(G,T)$ theorem to a conjecture of Weiss,
{\em Bull.\ Lond.\ Math.\ Soc.} {\bf 48} (2016),  12--18.

\bibitem{erik} E.\ Swartz,
Locally 3-Arc-Transitive Regular Covers of Complete Bipartite Graphs,
{\em Electronic J.\ Combin.} {\bf 23} (2016), Paper P2.18.

\bibitem{Rose}  S. Wilson, Rose Window graphs, {\it Ars Math. Contemp.} \textbf{1} (2008), 7--19.



\end{thebibliography}
\end{document}